\documentclass[hidelinks,onefignum,onetabnum]{siamart251216}

\newcommand{\KL}{\operatorname{KL}}

\newcommand{\TV}{\mathrm{TV}}
\newcommand{\JSD}{\operatorname{JSD}}

\usepackage{lipsum}
\usepackage{amsfonts}
\usepackage{graphicx}
\usepackage{epstopdf}
\usepackage{algorithmic}
\ifpdf
  \DeclareGraphicsExtensions{.eps,.pdf,.png,.jpg}
\else
  \DeclareGraphicsExtensions{.eps}
\fi

\newsiamremark{remark}{Remark}
\newsiamremark{hypothesis}{Hypothesis}
\crefname{hypothesis}{Hypothesis}{Hypotheses}
\newsiamthm{claim}{Claim}
\newsiamremark{fact}{Fact}
\crefname{fact}{Fact}{Facts}

\headers{Onsager Operators for Finite-State Bridges}
{A. Benabdallah and M. Dlala}

\title{Nonlocal Onsager Operators and Entropy Dissipation for Finite-State Schrödinger Bridges\thanks{Submitted to SIAM Journal on Mathematics of Data Science.}}

\author{
Abdallah Benabdallah\thanks{Higher Institute of Computer Science and
Multimedia of Sfax, University of Sfax, Sfax, Tunisia.
\email{abdellah.benabdallah@isims.usf.tn}}
\and
Mohsen Dlala\thanks{Department of Mathematics, Qassim University, Buraydah,
Saudi Arabia; and Department of Mathematics, Sfax Preparatory Engineering
Institute, University of Sfax, Sfax, Tunisia.
\email{3862@qu.edu.sa}}
}

\usepackage{amsopn}

\ifpdf
\hypersetup{
  pdftitle={An Example Article},
  pdfauthor={D. Doe, P. T. Frank, and J. E. Smith}
}
\fi

\begin{document}

\maketitle

\begin{abstract}
We investigate the Schrödinger bridge problem on a finite state space with a
strictly positive Markov reference kernel. Starting from the semi-dual convex
formulation, we introduce a gauge-fixed logarithmic flow on the terminal
Schrödinger potential and show that its equilibria coincide with the unique
solution of the bridge problem.

The induced terminal marginal evolution is the {\bf Schrödinger Bridge
Onsager Flow} (SBOF). This marginal equation is governed by a
state-dependent nonlocal Onsager operator, identified with the Hessian of the
semi-dual functional. We derive its associated Dirichlet form, establish
coercivity estimates on the appropriate gauge-fixed space, and interpret the
resulting equation as a nonlocal gradient-flow formulation of relative
entropy.

Under natural positivity assumptions, we prove global well-posedness of the
potential flow, convergence to the Schrödinger bridge, convergence of the
induced couplings and path measures, and exponential relaxation of the
terminal marginal. The latter follows from a uniform Poincaré inequality on
compact sublevel sets together with entropy--variance comparison estimates.
We also discuss the connection with finite-state generative modeling through
the Doob transform and illustrate the theory on finite-grid examples
involving rare states.
\end{abstract}


\section{Introduction}
\label{sec:Int}

Schrödinger bridge problems (SBP) provide a variational formulation for
selecting a stochastic evolution between prescribed endpoint distributions \cite{schrodinger1931umkehrung}.
Given a reference Markov process, the Schrödinger bridge (SB) is the path
measure of minimal relative entropy among all path laws with fixed initial
and terminal marginals \cite{jamison1975markov,follmer1985time}.  This variational problem connects
stochastic control, reciprocal processes, and entropic optimal transport
(EOT) through an entropy projection on path space
\cite{leonard2014survey,chen2016relation, nutz2021lectures}.

The finite-state setting deserves separate attention
\cite{jamison1975markov,georgiou2015positive,chen2016relation, jamison1975markov,pavon2010schrodinger,chow2021dynamical}. Although the dynamic
SBP reduces exactly to an EOT problem over endpoint couplings, a finite state
space does not come with a unique canonical differential structure: different choices of graph, generator, or transport metric lead to different
notions of gradients, divergences, mobilities, and curvature
\cite{maas2011gradient,erbar2012ricci,fathi2018curvature,
kamtue2024entropic}. Here we do not impose such a
structure externally; rather, we identify the operator structure generated by
the semi-dual parametrization of the finite-dimensional EOT problem.

Let $N \in \mathbb N^{*}$ and $S = \{1, \ldots, N\}$  a finite state space. Let \(Q \in \mathbb R^{N\times N}\) be the generator
of a reference continuous-time Markov chain (CTMC) on $S$. We denote by \(K_t=e^{tQ}\),
\(t\in[0,1]\), its Markov semigroup, and set \(K=K_1=e^Q\). Given strictly positive endpoint laws
\(\mu,\nu\in\mathcal P(S)\), the static SBP is
\begin{align}\label{eott}
\inf_{\pi\in\Pi(\mu,\nu)}
\KL(\pi\,\|\,\mu K).
\end{align}
The finite-dimensional entropy projection (\ref{eott}) is closely related to matrix
scaling, iterative proportional fitting (IPF), and Sinkhorn algorithms
\cite{sinkhorn1967,csiszar1975divergence,csiszar1984information,
ruschendorf1995,peyre2019computational}. For recent semigroup and
operator-theoretic perspectives on Sinkhorn stability, see
\cite{delmoral2026new}. The same problem admits a semi-dual
formulation in terms of a single terminal potential, as is standard in
EOT and recent potential-based SB
solvers
\cite{cuturi2019semidual,nutz2021lectures,leger2021gradient,
korotin2024light,gushchin2024light}.

The purpose of this paper is to identify the dissipative geometry generated
by this semi-dual formulation. For \(g:S\to\mathbb R\), define
\begin{equation}
\label{eq:intro-J}
\mathcal J(g)
=
\sum_{x\in S}\mu(x)
\log\!\left(
\sum_{y\in S}K(x,y)e^{g(y)}
\right)
-
\sum_{y\in S}\nu(y)g(y).
\end{equation}
Associated with \eqref{eq:intro-J} are the tilted conditional law and its
terminal marginal,
\begin{equation}
\label{eq:intro-pg-mg}
p_g(y|x)
=
\frac{K(x,y)e^{g(y)}}{\sum_{z\in S}K(x,z)e^{g(z)}},
\qquad
m_g(y)=\sum_{x\in S}\mu(x)p_g(y|x).
\end{equation}
A direct computation using \eqref{eq:intro-pg-mg} gives the fundamental
identity
\begin{equation}
\label{eq:intro-gradient}
\nabla\mathcal J(g)=m_g-\nu.
\end{equation}
Thus, by \eqref{eq:intro-gradient}, Schrödinger potentials solve the
nonlinear marginal equation \(m_g=\nu\), modulo the additive freedom
\(g\sim g+c\mathbf 1\). We study a continuous-time dynamics on terminal potentials
\(g\in\mathbb R^S\) that provides a gradient-flow-type relaxation of this
semi-dual marginal equation. 

The dynamics is posed on the gauge-fixed space
\begin{equation}
\label{eq:intro-gauge-space}
\mathcal H
=
\left\{
g\in\mathbb R^S:\mathbb E_\nu[g]=0
\right\}.
\end{equation}
On \(\mathcal H\), we introduce the following gauge-fixed logarithmic
potential flow:
\begin{equation}
\label{eq:intro-sbof}
\dot g_t
=
-\log\frac{m_{g_t}}{\nu}
+
\mathbb E_\nu\!\left[
\log\frac{m_{g_t}}{\nu}
\right]\mathbf 1 .
\end{equation}
The correction term in \eqref{eq:intro-sbof} preserves the gauge condition
in \eqref{eq:intro-gauge-space}. Moreover, the semi-dual functional is
dissipated along the potential flow:
\begin{equation}
\label{eq:intro-J-dissipation}
\frac{d}{dt}\mathcal J(g_t)
=
-\KL(m_{g_t}\|\nu)
-\KL(\nu\|m_{g_t})
\le 0.
\end{equation}
The dissipation identity \eqref{eq:intro-J-dissipation} shows that
\(\mathcal J\) is a Lyapunov functional for the potential dynamics (\ref{eq:intro-sbof}), up to the
irrelevant addition of a constant.

The central object of the paper is the nonlocal Onsager operator
\(\mathsf K_g\). It arises by linearizing the Schrödinger marginal map
\(g\mapsto m_g\), and equivalently as the Hessian of the semi-dual
functional:
\begin{equation}
\label{eq:intro-Kg}
\mathsf K_g
=
Dm_g
=
\nabla^2\mathcal J(g).
\end{equation}
The identity \eqref{eq:intro-Kg} is the starting point for the Onsager
interpretation developed below. Equivalently, for every \(f:S\to\mathbb R\),
\begin{equation}
\label{eq:intro-quadratic}
\langle f,\mathsf K_g f\rangle
=
\sum_{x\in S}
\mu(x)\operatorname{Var}_{p_g(\cdot|x)}(f).
\end{equation}
The covariance formula \eqref{eq:intro-quadratic} implies that
\(\mathsf K_g\) is symmetric positive semidefinite, annihilates constants,
and is strictly positive on \(\mathcal H\). The associated quadratic form can
be written as a nonlocal Dirichlet form on the terminal state space:
\begin{equation}
\label{eq:intro-dirichlet}
\langle f,\mathsf K_g f\rangle
=
\frac12
\sum_{y,z\in S}
A_g(y,z)(f(y)-f(z))^2,
\end{equation}
where
\begin{equation}
\label{Action}
A_g(y,z)
=
\sum_{x\in S}
\mu(x)p_g(y|x)p_g(z|x).
\end{equation}
Equations \eqref{eq:intro-dirichlet} and \eqref{Action} exhibit
\(\mathsf K_g\) as a nonlocal Dirichlet form with bridge-induced
conductances \(A_g(y,z)\).

This representation is reminiscent of the Dirichlet forms and transport
metrics used in discrete gradient-flow and entropic Ricci-curvature theories
for Markov chains
\cite{maas2011gradient,mielke2011geodesic,erbar2012ricci,
erbar2014gradient,fathi2016entropic,erbar2016ricci,erbar2017poincare}.
However, in the present setting the conductances \(A_g(y,z)\) are not
prescribed by an external graph or by a fixed reversible generator. They are
generated by the current Schrödinger conditional law \(p_g(\cdot|x)\) and
therefore depend on the terminal potential \(g\). Along the potential flow (\ref{eq:intro-sbof}),
the weights \(A_{g_t}(y,z)\) evolve with time, so the induced marginal
dynamics takes place on a bridge-dependent family of nonlocal weighted
graphs.

Passing from potentials to terminal marginals, set \(m_t=m_{g_t}\). By the
chain rule and the identity \(Dm_g=\mathsf K_g\), the potential flow
\eqref{eq:intro-sbof} induces the marginal evolution
\begin{equation}
\label{eq:intro-marginal}
\dot m_t
=
-\mathsf K_{g_t}
\log\frac{m_t}{\nu}.
\end{equation}
Since
\[
\frac{\delta}{\delta m}\KL(m\|\nu)
=
\log\frac{m}{\nu}+\mathbf 1,
\]
and since \(\mathsf K_{g_t}\mathbf 1=0\), this can be written in the
Onsager form
\begin{align}\label{SBOFFFF}
\dot m_t
=
-\mathsf K_{g_t}
\frac{\delta}{\delta m}\KL(m_t\|\nu).
\end{align}
We refer to this induced marginal equation (\ref{SBOFFFF}) as the Schrödinger Bridge
Onsager Flow (SBOF). Thus the terminal marginal evolves as an
entropy-dissipating gradient dynamics with respect to the bridge-induced,
state-dependent mobility \(\mathsf K_{g_t}\).

 Consequently,
\begin{align}\label{eq:intro-entropy-dissipation}
\frac{d}{dt}\KL(m_t\|\nu)
=&
-
\left\langle
\log\frac{m_t}{\nu},
\mathsf K_{g_t}
\log\frac{m_t}{\nu}
\right\rangle, \notag\\
=&-
\mathcal E_{g_t}\!\left(\log\frac{m_t}{\nu}\right).
\end{align}
The Dirichlet form \(\mathcal E_g\) therefore plays the role of the
nonlocal Fisher information, or entropy-production functional, associated
with the induced SBOF (\ref{SBOFFFF}).

This makes the time-dependent family \(\{\mathsf K_{g_t}\}_{t\ge0}\) a
geometric diagnostic of the SBP. It encodes, along the relaxation path, the
interaction between the reference dynamics \(Q\), the endpoint laws
\(\mu,\nu\), and the current Schrödinger potential. Thus the relevant
conditioning is not only spectral at a fixed time; it is governed by the
evolution of the bridge-induced mobility itself. The present paper
establishes the operator-theoretic and dissipative structure of this
time-varying Onsager geometry. A quantitative characterization of when the
induced geometry facilitates or obstructs transport between \(\mu\) and
\(\nu\) is left for future work. In this sense, the exponential rate obtained
below should be viewed as a coarse but rigorous measure of the conditioning
of the induced Onsager geometry along the potential flow trajectory.

The compactness of semi-dual sublevel sets on \(\mathcal H\) gives uniform
control of this geometry along the flow. For every \(g_0\in\mathcal H\), the
trajectory remains in
\begin{align*}
\mathcal C_{g_0}
=
\{g\in\mathcal H:\mathcal J(g)\le \mathcal J(g_0)\}.
\end{align*}
On the compact set $\mathcal C_{g_0}$, the forms \(\mathcal E_g\) satisfy a uniform Poincaré
inequality on \(\nu\)-mean-zero functions. Combined with an
entropy--variance comparison for \(\log(m_g/\nu)\), this yields
\begin{align*}
\KL(m_t\|\nu)
\le
\KL(m_0\|\nu)
\exp(-\omega_{g_0}t),
\end{align*}
for some \(\omega_{g_0}>0\) depending only on the initial sublevel set.

Finally, the terminal potential reconstructs the full dynamic bridge through
a space-time Doob transform
\cite{doob1957conditional,jamison1975markov,leonard2014survey}. For each
\(g\), define the space-time harmonic function
\begin{equation}
\label{eq:intro-doob-h}
\varphi_s^g
=
e^{(1-s)Q}e^g,
\qquad s\in[0,1].
\end{equation}
The corresponding Doob-transformed generator is given, for \(x\neq y\), by
\begin{equation}
\label{eq:intro-doob-generator}
Q_s^g(x,y)
=
Q(x,y)
\frac{\varphi_s^g(y)}{\varphi_s^g(x)},
\end{equation}
with diagonal entries chosen so that rows sum to zero. The function
\(\varphi_s^g\) in \eqref{eq:intro-doob-h} defines the space-time harmonic
tilt, and \eqref{eq:intro-doob-generator} is the associated
time-inhomogeneous Markov generator. At the limiting potential \(g^\star\),
this Doob transform realizes the SB path measure. Thus convergence of the
terminal potential implies convergence of the endpoint coupling and of the
reconstructed path measure.

\paragraph{Contributions.}
The main contributions are:
\begin{enumerate}
    \item We introduce a gauge-fixed logarithmic flow on terminal
    Schrödinger potentials and prove its global well-posedness, dissipation
    of the semi-dual functional, and convergence to the unique gauge-fixed
    Schrödinger potential. The induced terminal marginal dynamics is the
    Schrödinger Bridge Onsager Flow (SBOF).

    \item We identify the intrinsic Onsager operator
    \(\mathsf K_g=Dm_g=\nabla^2\mathcal J(g)\), derive its covariance
    representation, and show that it induces a bridge-dependent nonlocal
    Dirichlet form with conductances \(A_g(y,z)\).

    \item We show that the SBOF has the Onsager form
    \[
    \dot m_t
    =
    -\mathsf K_{g_t}
    \frac{\delta}{\delta m}\KL(m_t\|\nu),
    \]
    and derive the marginal entropy-dissipation identity
    \[
    \frac{d}{dt}\KL(m_t\|\nu)
    =
    -\mathcal E_{g_t}\!\left(\log\frac{m_t}{\nu}\right).
    \]
    This identifies the Dirichlet form as the nonlocal Fisher information of
    the induced marginal flow.

    \item We establish a uniform Poincaré inequality on compact semi-dual
    sublevel sets and prove exponential convergence of the terminal marginal
    in relative entropy.

    \item We show that the limiting potential reconstructs the dynamic SB
    through a Doob transform of the reference CTMC.
\end{enumerate}
The paper is organized as follows. Section~\ref{sec:finite-sb} recalls the
finite-state SBP, its dynamic-to-static reduction, and the semi-dual
formulation. Section~\ref{sec:sbof} introduces the gauge-fixed logarithmic
potential flow on terminal Schrödinger potentials and proves global
well-posedness and convergence. Section~\ref{sec:onsager-geometry} derives
the induced SBOF on terminal marginals and
identifies the Onsager operator \(\mathsf K_g\), the associated nonlocal
Dirichlet form, and the marginal entropy-dissipation identity.
Section~\ref{sec:exponential} proves the uniform Poincaré inequality and
exponential relaxation. Finally, Section~\ref{sec:dynamic-reconstruction}
discusses the Doob-transform reconstruction and presents numerical
illustrations.

\section{Finite-State Schrödinger Bridges and Semi-Duality}
\label{sec:finite-sb}

We recall the finite-state SB framework used throughout the paper. Starting
from a CTMC reference on a finite state space, the dynamic entropy-minimization
problem on path space reduces exactly to a static EOT problem over endpoint
couplings \cite{leonard2014survey,chen2016relation}. This finite-dimensional
reduction is the starting point for the semi-dual formulation below. The
terminal-potential parametrization defines a nonlinear marginal map
\(g\mapsto m_g\), whose differential
\(\mathsf K_g=Dm_g=\nabla^2\mathcal J(g)\) will provide the Onsager mobility
studied in the rest of the paper.

The finite-dimensional static problem is classically related to IPF, Sinkhorn
scaling, matrix scaling, and alternating information projections
\cite{sinkhorn1967,ruschendorf1995,csiszar1975divergence,
csiszar1984information,cuturi2013sinkhorn,peyre2019computational}. We use
this background to introduce the endpoint parametrization and the associated
semi-dual functional.

\subsection{Static reduction and endpoint parametrization}
\label{subsec:static-parametrization}

Let \(N\) be a positive integer, let \(S=\{1,\ldots,N\}\) be a finite state
space, and let \(Q\) be the generator of a reference CTMC on \(S\). Let \(K_t=e^{tQ}\), \(t\in[0,1]\), denote the associated
Markov semigroup, and set \(K=K_1=e^Q\). Throughout the paper we assume
\begin{equation}
\label{eq:positivity}
K(x,y)>0,\qquad \mu(x)>0,\qquad \nu(y)>0,
\qquad \forall x,y\in S .
\end{equation}
Let \(\Omega=D([0,1],S)\) be the Skorokhod space of càdlàg paths with values
in \(S\), and let \(\mathbf R_\mu\) denote the law of the reference process
initialized with \(X_0\sim\mu\). Its endpoint law is
\begin{align*}
\mathbf R_\mu(X_0=x,X_1=y)=\mu(x)K(x,y),
\qquad \forall x,y\in S.
\end{align*}

The dynamic SBP is
\begin{equation}
\label{eq:dynamic-sb}
\inf_{\mathbf P\in\mathcal A(\mu,\nu)}
\KL(\mathbf P\,\|\,\mathbf R_\mu),
\end{equation}
where
\begin{align*}
\mathcal A(\mu,\nu)
=
\left\{
\mathbf P\in\mathcal P(\Omega):
\mathbf P_0=\mu,\;
\mathbf P_1=\nu,\;
\mathbf P\ll \mathbf R_\mu
\right\}.
\end{align*}
Here \(\mathbf P_t\) denotes the time-\(t\) marginal law of \(\mathbf P\),
and \(\mathbf P\ll \mathbf R_\mu\) denotes absolute continuity with respect
to the reference path measure. Under the positivity assumption \eqref{eq:positivity}, the admissible set is
nonempty and the dynamic SBP \eqref{eq:dynamic-sb} admits a unique minimizer
\(\mathbf P^\star\), called the Schrödinger bridge between \(\mu\) and
\(\nu\) relative to the reference path measure \(\mathbf R_\mu\). 

By the chain rule for relative entropy, a standard disintegration argument
reduces \eqref{eq:dynamic-sb} to the static problem
\begin{equation}
\label{eq:static-sb}
\inf_{\pi\in\Pi(\mu,\nu)}
\KL(\pi\,\|\,\mu K),
\end{equation}
where \((\mu K)(x,y)=\mu(x)K(x,y)\), and
\begin{align*}
\Pi(\mu,\nu)
=
\left\{
\pi\in\mathcal P(S\times S):
\pi_X=\mu,\;
\pi_Y=\nu
\right\}.
\end{align*}
The set \(\Pi(\mu,\nu)\) is the transportation polytope, or coupling set,
between \(\mu\) and \(\nu\).
Here \(\pi_X\) and \(\pi_Y\) denote the first and second marginals of
\(\pi\), namely
\begin{align*}
\pi_X(x)=\sum_{y\in S}\pi(x,y),
\qquad
\pi_Y(y)=\sum_{x\in S}\pi(x,y).
\end{align*}
The static formulation \eqref{eq:static-sb} is the finite-dimensional
EOT problem associated with the reference
kernel \(K\); see, for example,
\cite{leonard2014survey,chen2016relation,peyre2019computational}. The optimal path measure is recovered from the optimal endpoint coupling by
disintegration with respect to the reference Markov bridges. More precisely,
if \(\pi^\star\) solves \eqref{eq:static-sb}, then
\begin{equation}
\label{eq:path-reconstruction}
\mathbf P^\star(d\omega)
=
\sum_{x,y\in S}
\pi^\star(x,y)\,
\mathbf R_\mu(d\omega\mid X_0=x,X_1=y).
\end{equation}
The reconstruction formula \eqref{eq:path-reconstruction} shows that the
dynamic bridge is obtained by combining the optimal endpoint coupling with
the reference conditional bridges. Equivalently, the optimal coupling has the
classical Schrödinger scaling form
\begin{equation}
\label{eq:schrodinger-scaling}
\pi^\star(x,y)
=
f^\star(x)K(x,y)g^\star(y),
\end{equation}
where \(f^\star,g^\star>0\) solve
\begin{equation}
\label{eq:schrodinger-system}
f^\star(x)\sum_{y\in S}K(x,y)g^\star(y)=\mu(x),
\qquad
g^\star(y)\sum_{x\in S}f^\star(x)K(x,y)=\nu(y).
\end{equation}
The system \eqref{eq:schrodinger-system} enforces the prescribed marginals
of the scaled coupling \eqref{eq:schrodinger-scaling}. The classical
IPF/Sinkhorn iterations alternately impose these two constraints:
\begin{equation}
\label{eq:sinkhorn-iterations}
f^{k+1}(x)
=
\frac{\mu(x)}
{\sum_{y\in S}K(x,y)g^k(y)},
\qquad
g^{k+1}(y)
=
\frac{\nu(y)}
{\sum_{x\in S}f^{k+1}(x)K(x,y)}.
\end{equation}
Thus \eqref{eq:sinkhorn-iterations} is the alternating scaling procedure
associated with the Schrödinger system \eqref{eq:schrodinger-system}.
Under the positivity assumption \eqref{eq:positivity},  iterations (\ref{eq:sinkhorn-iterations}) 
converge to the scaling factors, up to the usual multiplicative gauge
\cite{sinkhorn1967,ruschendorf1995,csiszar1975divergence,
csiszar1984information,peyre2019computational}.

In the present paper, we use an equivalent one-potential parametrization in
which the first marginal is fixed by construction. This is the standard
semi-dual viewpoint in EOT, and it is also
used in potential-based formulations of SB solvers
\cite{peyre2019computational,nutz2021lectures,leger2021gradient,
korotin2023light}.

For \(g:S\to\mathbb R\), set
\begin{equation}
\label{eq:Zg-pg}
Z_g(x)=\sum_{z\in S}K(x,z)e^{g(z)},
\qquad
p_g(y|x)=\frac{K(x,y)e^{g(y)}}{Z_g(x)}.
\end{equation}
Using \eqref{eq:Zg-pg}, define the associated endpoint coupling by
\begin{equation}\label{J}
\pi_g(x,y)
=
\mu(x)p_g(y|x)
=
\mu(x)\frac{K(x,y)e^{g(y)}}{Z_g(x)}.
\end{equation}
It satisfies \((\pi_g)_X=\mu\). Its terminal marginal is
\begin{equation}
\label{eq:mg}
m_g(y)
=
\sum_{x\in S}\pi_g(x,y)
=
\sum_{x\in S}
\mu(x)\frac{K(x,y)e^{g(y)}}{Z_g(x)}.
\end{equation} 
Thus, by \eqref{eq:mg}, the SB condition is equivalent to the nonlinear
marginal equation
\[
m_g=\nu .
\]
The parametrization is invariant under additive constants:
\[
\pi_{g+c\mathbf 1}=\pi_g,
\qquad
m_{g+c\mathbf 1}=m_g,
\qquad c\in\mathbb R.
\]
This is why we work on the gauge-fixed subspace
\begin{equation}
\label{eq:gauge-space}
\mathcal H
=
\left\{
g\in\mathbb R^S:
\mathbb E_\nu[g]=0
\right\}.
\end{equation}
The subspace \(\mathcal H\) in \eqref{eq:gauge-space} fixes the additive
freedom of the terminal potential.

\subsection{Semi-dual formulation}
\label{subsec:semi-dual-formulation}

We now record the semi-dual formulation associated with the
one-potential parametrization. For \(g:S\to\mathbb R\), define
\begin{align}
\mathcal J(g)
=& \sum_{x\in S}\mu(x)
\log\!\left(
\sum_{y\in S}K(x,y)e^{g(y)}
\right)
-
\sum_{y\in S}\nu(y)g(y),\notag\\
=&\sum_{x\in S}\mu(x)\log Z_g(x)
-
\mathbb E_\nu[g].\label{eq:J}
\end{align}
The functional \(\mathcal J\) is smooth, convex, and invariant under
additive constants. It is the negative of the usual concave semi-dual
objective for the EOT problem
\eqref{eq:static-sb}.

\begin{theorem}
\label{thm:finite-duality}
Assume that \eqref{eq:positivity} holds. Then
\begin{equation}
\label{eq:dual-relation}
\inf_{\pi\in\Pi(\mu,\nu)}
\KL(\pi\,\|\,\mu K)
=
-\inf_{g\in\mathbb R^S}\mathcal J(g).
\end{equation}
Moreover, if \(g^\star\) minimizes \(\mathcal J\), then the coupling
\(\pi_{g^\star}\) defined in \eqref{eq:Zg-pg} solves the static
SBP \eqref{eq:static-sb}.
\end{theorem}

The proof follows from finite-dimensional entropy duality; details are given
in the supplementary material. The
sign convention in \eqref{eq:J} is chosen so that the semi-dual problem is a
minimization problem. Relation \eqref{eq:dual-relation} reduces the static SBP
to minimizing the smooth convex functional \(\mathcal J\) over \(\mathbb R^S\).
Because \(\mathcal J\) is invariant under additive constants, its minimizers
are characterized by the first-order optimality condition on the quotient by
constants. The next proposition records the gradient of \(\mathcal J\); its
vanishing is precisely the Schr\"odinger marginal equation.

\begin{proposition}
\label{prop:gradient-J}
For every \(g\in\mathbb R^S\),
\begin{equation}
\label{eq:grad-J}
\nabla \mathcal J(g)=m_g-\nu.
\end{equation}
\end{proposition}

\begin{proof}
For \(y\in S\),
\begin{align*}
\frac{\partial \mathcal J}{\partial g(y)}(g)
=
\sum_{x\in S}
\mu(x)
\frac{K(x,y)e^{g(y)}}{\sum_{z\in S}K(x,z)e^{g(z)}}
-
\nu(y).
\end{align*}
By \eqref{eq:Zg-pg} and \eqref{eq:mg}, the right-hand side is exactly
\(m_g(y)-\nu(y)\).
\end{proof}

Thus the first-order optimality condition for \(\mathcal J\) is precisely
the Schr\"odinger marginal equation \(m_g=\nu\).

\subsection{Gauge fixing and static well-posedness}
\label{subsec:gauge-static-theory}

Since $\mathcal J$ is invariant under $g\mapsto g+c\mathbf 1$, uniqueness
holds only modulo additive constants. We fix the gauge by working on the
hyperplane $\mathcal H$ defined in \eqref{eq:gauge-space}.

This choice is compatible with the logarithmic flow introduced in the next
section.

\begin{theorem}
\label{thm:exist-unique}
Assume that \eqref{eq:positivity} holds. Then the semi-dual functional
$\mathcal J$ is convex on $\mathbb R^S$, invariant under additive
constants, and coercive on the gauge-fixed space $\mathcal H$. Moreover,
$\mathcal J$ is strictly convex on $\mathcal H$. Consequently,
$\mathcal J$ admits a unique minimizer $g^\star\in\mathcal H$, and this
minimizer satisfies
\[
m_{g^\star}=\nu.
\]
Thus $g^\star$ is the unique gauge-fixed Schr\"odinger potential, and
$\pi_{g^\star}$ is the unique solution of the static problem
\eqref{eq:static-sb}.
\end{theorem}

The proof is given in Appendix~\ref{app:static-theory}. We briefly indicate
the main points. The Hessian of $\mathcal J$ is the differential of the
marginal map $g\mapsto m_g$:
\begin{equation}
\label{eq:hessian-static-preview}
\nabla^2\mathcal J(g)=Dm_g.
\end{equation}
More explicitly, for every $f:S\to\mathbb R$,
\begin{equation}
\label{eq:hessian-cov-preview}
\left\langle f,\nabla^2\mathcal J(g)f\right\rangle
=
\sum_{x\in S}
\mu(x)\operatorname{Var}_{p_g(\cdot|x)}(f).
\end{equation}
Hence $\mathcal J$ is convex. Since $K(x,y)>0$, each
$p_g(\cdot|x)$ has full support, and the right-hand side of
\eqref{eq:hessian-cov-preview} vanishes if and only if $f$ is constant.
Thus the Hessian is strictly positive on the gauge-fixed
hyperplane $\mathcal H$.

Furthermore, $\mathcal J$ is coercive on $\mathcal H$. Therefore it
attains a unique minimizer $g^\star\in\mathcal H$. By the gradient identity
\eqref{eq:grad-J}, the first-order optimality condition gives
$m_{g^\star}=\nu$.

The covariance representation \eqref{eq:hessian-cov-preview} will be used
again in Section~\ref{sec:onsager-geometry}, where the Hessian is identified
with the nonlocal Onsager operator governing the induced marginal dynamics.

\section{Potential Dynamics and Convergence to the Schr\"odinger Potential}
\label{sec:sbof}

We now turn the static marginal equation $m_g=\nu$ into a continuous-time
dynamics on the terminal potential. This perspective belongs to a recent
variational literature on Sinkhorn and IPF-type methods, in which entropic
scaling algorithms are interpreted as gradient, Bregman-gradient, or
mirror-flow dynamics on dual variables
\cite{leger2021gradient,mishchenko2019sinkhorn,aubin2022mirror,
karimi2024sinkhornflow,srinivasan2025designing}. Logarithmic flows on the 
dual potential, with or without gauge fixing, have been studied in this context
\cite{leger2021gradient,karimi2024sinkhornflow}. Related continuous-time 
perspectives also appear in recent SB flow formulations 
\cite{debortoli2024sbf}.

The purpose of this section is different. We introduce a gauge-fixed 
logarithmic potential flow whose \emph{induced marginal dynamics} reveals
a nonlocal Onsager geometry. The operator
$\mathsf K_{g_t}=Dm_{g_t}=\nabla^2\mathcal J(g_t)$, which governs the
marginal velocity $\dot m_t$, is not prescribed by an external metric or
mirror map. It is generated by the current Schr\"odinger potential itself,
and its associated Dirichlet form encodes the bridge-dependent
conductances $A_{g_t}(y,z)$. This induced geometry is the central object
of the paper; the potential flow is merely the lens through which it is
revealed.
\subsection{Definition and equilibria}
\label{subsec:flow-definition}

Let \(g_0\in\mathcal H\). We define the gauge-fixed logarithmic potential flow by
\begin{equation}
\label{eq:SBOF}
\dot g_t
=
-\log\frac{m_{g_t}}{\nu}
+
\mathbb E_\nu\!\left[
\log\frac{m_{g_t}}{\nu}
\right]\mathbf 1,
\end{equation}
with initial condition $g_{t=0}=g_0.$ The logarithm is well defined for all \(g\in\mathbb R^S\), since the
positivity assumption \eqref{eq:positivity} implies \(m_g(y)>0\) for every
\(y\in S\). The second term in \eqref{eq:SBOF} fixes the additive gauge. In
fact,
\begin{align*}
\frac{d}{dt}\mathbb E_\nu[g_t]
=
-\mathbb E_\nu\!\left[
\log\frac{m_{g_t}}{\nu}
\right]
+
\mathbb E_\nu\!\left[
\log\frac{m_{g_t}}{\nu}
\right]
=0.
\end{align*}
Hence, if \(g_0\in\mathcal H\), then \(g_t\in\mathcal H\) for as long as the
solution exists.

The equilibria of \eqref{eq:SBOF} are precisely the gauge-fixed
Schrödinger potentials. Indeed, if \(g\in\mathcal H\) is an equilibrium, then
\begin{align*}
\log\frac{m_g}{\nu}
=
\mathbb E_\nu\!\left[
\log\frac{m_g}{\nu}
\right]\mathbf 1,
\end{align*}
so \(m_g/\nu\) is constant on \(S\). Since both \(m_g\) and \(\nu\) are
probability measures, this constant must be one, and therefore \(m_g=\nu\).
Conversely, if \(m_g=\nu\), then the right-hand side of \eqref{eq:SBOF}
vanishes. By Theorem~\ref{thm:exist-unique}, the unique equilibrium in
\(\mathcal H\) is the Schrödinger potential \(g^\star\).

\subsection{Lyapunov identity and global convergence}
\label{subsec:flow-convergence}

We next record the basic well-posedness and convergence properties of the
potential flow. The key point is that the semi-dual functional
\(\mathcal J\) is a Lyapunov function whose dissipation is exactly the
symmetrized relative entropy between the current terminal marginal and the
target.

\begin{theorem}
\label{thm:SBOF-convergence}
Assume that \eqref{eq:positivity} holds. For every \(g_0\in\mathcal H\),
the potential flow \eqref{eq:SBOF} admits a unique global solution
\(g_t\in\mathcal H\), \(t\ge0\). Moreover,
\begin{equation}
\label{eq:J-dissipation}
\frac{d}{dt}\mathcal J(g_t)
=
-\KL(m_{g_t}\,\|\,\nu)
-\KL(\nu\,\|\,m_{g_t})
\le 0.
\end{equation}
There exists a unique \(g^\star\in\mathcal H\) such that
\(m_{g^\star}=\nu\). Moreover, \(g_t\to g^\star\) and
\(m_{g_t}\to\nu\) as \(t\to\infty\).
\end{theorem}

\begin{proof}
The full proof is given in Appendix~\ref{app:sbof-convergence}. We only
derive the dissipation identity here. Since
\(\nabla\mathcal J(g)=m_g-\nu\), along a smooth solution of \eqref{eq:SBOF} we have
\begin{align*}
\frac{d}{dt}\mathcal J(g_t)
=
\left\langle m_{g_t}-\nu,\dot g_t\right\rangle .
\end{align*}
The gauge correction in \eqref{eq:SBOF} does not contribute, because
\(\left\langle m_{g_t}-\nu,\mathbf 1\right\rangle=0\).
Therefore
\begin{align*}
\frac{d}{dt}\mathcal J(g_t)
&=
-\sum_{y\in S}
\bigl(m_{g_t}(y)-\nu(y)\bigr)
\log\frac{m_{g_t}(y)}{\nu(y)}\\
&=
-\KL(m_{g_t}\,\|\,\nu)-\KL(\nu\,\|\,m_{g_t}),
\end{align*}
which gives \eqref{eq:J-dissipation}.
\end{proof}

The identity \eqref{eq:J-dissipation} implies that every trajectory remains
in the sublevel set
\begin{align*}
\mathcal C_{g_0}
=
\left\{
g\in\mathcal H:\mathcal J(g)\le \mathcal J(g_0)
\right\}.
\end{align*}
Since \(\mathcal J\) is coercive on \(\mathcal H\), the  set $\mathcal C_{g_0}$ is compact.
This compactness prevents finite-time blow-up and is also used below to
obtain uniform functional inequalities for the marginal dynamics.

\subsection{Convergence of endpoint couplings and path measures}
\label{subsec:coupling-path-convergence}
The convergence $g_t\to g^\star$ established in
Theorem~\ref{thm:SBOF-convergence} carries over to the endpoint
couplings $\pi_{g_t}$ and the reconstructed path measures
$\mathbf P^{g_t}$. Recall that
for each \(g\in\mathbb R^S\),
\begin{align}\label{pig}
\pi_g(x,y)
=
\mu(x)
\frac{K(x,y)e^{g(y)}}{Z_g(x)}.
\end{align}
The corresponding path measure is obtained by lifting \(\pi_g\) through the
reference Markov bridges:
\begin{equation}
\label{eq:Pg-definition}
\mathbf P^g(d\omega)
=
\sum_{x,y\in S}
\pi_g(x,y)\,
\mathbf R_\mu(d\omega\mid X_0=x,X_1=y).
\end{equation}

\begin{corollary}
\label{cor:coupling-path-convergence}
Let $g_t$ solve the potential flow \eqref{eq:SBOF} with $g_0\in\mathcal H$,
and let $\pi_{g_t}$ and $\mathbf P^{g_t}$ be defined by \eqref{pig} and
\eqref{eq:Pg-definition}, respectively. Then
\begin{align}
\|\pi_{g_t}-\pi_{g^\star}\|_{\TV}\to0,
\qquad
\|\mathbf P^{g_t}-\mathbf P^{g^\star}\|_{\TV}\to0,
\qquad t\to\infty.
\end{align}
\end{corollary}

The proof is given in Appendix~\ref{app:path-convergence}. Since
\(g_t\to g^\star\) and the map \(g\mapsto\pi_g\) is smooth, the endpoint
couplings converge. The convergence of path measures then follows from the
linear reconstruction formula \eqref{eq:Pg-definition}. Thus convergence of the potential flow determines  convergence of the
full SB object through the endpoint disintegration formula.

\section{The Intrinsic Onsager Operator}
\label{sec:onsager-geometry}

We now identify the intrinsic geometry behind the induced marginal dynamics
$m_t=m_{g_t}$, hereafter the Schr\"odinger Bridge Onsager Flow (SBOF).
Logarithmic potential flows on dual variables have appeared in recent
Sinkhorn-type algorithms
\cite{leger2021gradient,mishchenko2019sinkhorn,aubin2022mirror,
karimi2024sinkhornflow}. The novel structure here is the Onsager geometry
generated by the Schr\"odinger semi-dual parametrization itself.
Differentiating the marginal map $g\mapsto m_g$ produces the operator
\begin{equation}
\label{eq:Kg-intro-onsager}
\mathsf K_g
=
Dm_g
=
\nabla^2\mathcal J(g),
\end{equation}
which acts as the Onsager mobility in the marginal equation
\begin{equation}
\label{eq:SBOF-marginal}
\dot m_t
=
-\mathsf K_{g_t}
\log\frac{m_t}{\nu}.
\end{equation}
We will show below that \eqref{eq:SBOF-marginal} is a nonlocal gradient flow
of the relative entropy with respect to the bridge-dependent mobility
$\mathsf K_{g_t}$.

In this sense, \(\mathsf K_g\) encodes the effective geometry of the
finite-state SB. Its coefficients depend on the reference
kernel \(K=e^Q\), the source law \(\mu\), and the current Schrödinger
potential \(g\), while its coercivity properties determine the conditioning
and relaxation of the marginal flow (\ref{eq:SBOF-marginal}). Equivalently, the associated Dirichlet
form encodes, through the current bridge conditional law, the directions in
terminal marginal space that are efficiently corrected by the bridge
dynamics, and the directions in which the bridge is ill-conditioned. Thus
\(\mathsf K_g\) is not merely the Hessian of the semi-dual functional; it is
the bridge-induced Onsager geometry through which the entropy force is
converted into marginal motion \cite{onsager1931reciprocal,otto2001geometry,mielke2011geodesic}.

\subsection{Differential of the marginal map}
\label{subsec:differential-marginal-map}

We first compute the differential of the map \(g\mapsto m_g\). Recall that
\begin{align*}
p_g(y|x)
=
\frac{K(x,y)e^{g(y)}}{Z_g(x)},
\qquad
Z_g(x)=\sum_{z\in S}K(x,z)e^{g(z)},
\end{align*}
and
\begin{align*}
m_g(y)=\sum_{x\in S}\mu(x)p_g(y|x).
\end{align*}

\begin{proposition}
\label{prop:Dmg}
Assume that \eqref{eq:positivity} holds. The map \(g\mapsto m_g\) is \(C^\infty\) on \(\mathbb R^S\). For every
\(g\in\mathbb R^S\), its differential
\(Dm_g:\mathbb R^S\to\mathbb R^S\) is the linear operator 
\begin{equation}
\label{eq:Kg}
\mathsf K_g
:=
Dm_g
=\nabla^2\mathcal J(g)=
\sum_{x\in S}\mu(x)
\left[
\operatorname{Diag}(p_g(\cdot|x))
-
p_g(\cdot|x)p_g(\cdot|x)^\top
\right].
\end{equation}
Equivalently, for every \(h\in \mathbb R^S\),
\begin{equation}
\label{eq:Dmg-action}
(Dm_g[h])(y)
=
\sum_{x\in S}
\mu(x)p_g(y|x)
\left(
h(y)-\sum_{z\in S}p_g(z|x)h(z)
\right),
\qquad \forall y\in S.
\end{equation}
\end{proposition}

\begin{proof}
Since \(Z_g(x)>0\) for all \(x\), the map \(g\mapsto p_g(\cdot|x)\) is
\(C^\infty\), and so is \(g\mapsto m_g\). Let \(h:S\to\mathbb R\). Differentiating \(Z_g(x)\) in the direction \(h\) gives
\[
DZ_g(x)[h]
=
\sum_{z\in S}K(x,z)e^{g(z)}h(z)
=
Z_g(x)\sum_{z\in S}p_g(z|x)h(z).
\]
Applying the quotient rule to
\(
p_g(y|x)
\),
then yields
\[
Dp_g(y|x)[h]
=
p_g(y|x)
\left(
h(y)-\sum_{z\in S}p_g(z|x)h(z)
\right).
\]
Summing over \(x\) with weights \(\mu(x)\) gives
\eqref{eq:Dmg-action}. In matrix form, this is \eqref{eq:Kg}. Finally, by
Proposition~\ref{prop:gradient-J},
$
\nabla\mathcal J(g)=m_g-\nu.$ Differentiating once more yields $\nabla^2\mathcal J(g)=Dm_g=\mathsf K_g$.
The matrix representation \eqref{eq:Kg} follows immediately from the
directional formula \eqref{eq:Dmg-action}. This completes the proof.

\end{proof}

The operator \(\mathsf K_g\) is symmetric and positive semidefinite. Indeed,
for every \(f:S\to\mathbb R\),
\begin{equation}
\label{eq:fKgf}
\langle f,\mathsf K_g f\rangle
=
\sum_{x\in S}
\mu(x)
\operatorname{Var}_{p_g(\cdot|x)}(f).
\end{equation}
In particular, $\mathsf K_g\mathbf 1=0.$
Under the positivity assumption \eqref{eq:positivity}, each
\(p_g(\cdot|x)\) has full support. Hence the right-hand side of
\eqref{eq:fKgf} vanishes if and only if \(f\) is constant on \(S\). Thus
\(\mathsf K_g\) is strictly positive on the gauge-fixed  subspace
\(\mathcal H\).

The covariance representation of $\mathsf K_g$ induces a nonlocal
Dirichlet form on the terminal state space. For $g\in\mathbb R^S$, define
\begin{equation}
\label{eq:dirichlet-energy}
\mathcal E_g(f)
:=
\langle f,\mathsf K_g f\rangle
=
\sum_{x\in S}
\mu(x)\operatorname{Var}_{p_g(\cdot|x)}(f),
\qquad
f\in\mathbb R^S.
\end{equation}
Equivalently, using \eqref{eq:dirichlet-energy} and the identity
\[
\operatorname{Var}_p(f)
=
\frac12
\sum_{y,z\in S}
p(y)p(z)(f(y)-f(z))^2,
\]
we obtain the nonlocal Dirichlet representation
\begin{equation}
\label{eq:dirichlet-form}
\mathcal E_g(f)
=
\frac12
\sum_{y,z\in S}
A_g(y,z)(f(y)-f(z))^2,
\end{equation}
with bridge-induced conductances
\begin{equation}
\label{eq:conductances}
A_g(y,z)
=
\sum_{x\in S}
\mu(x)p_g(y|x)p_g(z|x).
\end{equation}

The coefficients \(A_g(y,z)\) are symmetric and nonnegative. Under
\eqref{eq:positivity}, they are strictly positive for all \(y,z\in S\).
Thus $\mathcal E_g$ is the Dirichlet form of a nonlocal jump operator on
the complete weighted graph over the terminal state space.

The weights $A_g(y,z)$ should be viewed as bridge-induced conductances:
they are not prescribed by a fixed terminal graph or by a fixed reversible
generator, but are generated by the current tilted bridge kernel
$p_g(\cdot|x)$. In particular, terminal states $y$ and $z$ are strongly
coupled when they have large joint conditional weight from the same initial
state under the tilted reference dynamics. Along the potential dynamic trajectory, these
conductances become time-dependent through
\[
A_{g_t}(y,z)
=
\sum_{x\in S}
\mu(x)p_{g_t}(y|x)p_{g_t}(z|x),
\]
so that the induced Onsager geometry evolves with the potential flow itself.

\begin{remark}
\label{rem:bridge-curvature}
The time-dependent family $\{\mathsf K_{g_t}\}_{t\ge0}$ defines a
\emph{non-autonomous} Riemannian metric on the probability simplex
$\mathcal P(S)$, whose metric tensor at $m_{g_t}$ is the inverse of the
Onsager mobility $\mathsf K_{g_t}$ restricted to the tangent space
$T_{m_{g_t}}\mathcal P(S)=\{v\in\mathbb R^S:\sum_y v(y)=0\}$.
In classical discrete gradient-flow theory \cite{maas2011gradient,
erbar2012ricci}, the metric is fixed and curvature is studied via
Bochner inequalities for the associated Dirichlet form. Here, the metric
itself evolves with the Schr\"odinger potential, raising the question of
whether a \emph{dynamical curvature} can be defined for the family
$\mathcal E_{g_t}$. A positive answer would yield quantitative
acceleration estimates for the SBOF beyond the uniform Poincar\'e
inequality of Section~\ref{sec:exponential}. We leave this direction for
future work.
\end{remark}

\subsection{Marginal Onsager dynamics and entropy dissipation}
\label{subsec:marginal-onsager-dissipation}

We now pass from the potential evolution to the induced terminal marginal
evolution. Let $g_t$ solve the potential flow \eqref{eq:SBOF} with
$g_0\in\mathcal H$, and set $m_t=m_{g_t}$. By the chain rule and
Proposition~\ref{prop:Dmg},
\[
\dot m_t=Dm_{g_t}[\dot g_t]=\mathsf K_{g_t}\dot g_t.
\]
Using \eqref{eq:SBOF} and $\mathsf K_{g_t}\mathbf 1=0$, we obtain
\begin{equation}
\label{eq:dotmt}
\dot m_t=-\mathsf K_{g_t}\log\frac{m_t}{\nu}.
\end{equation}

Since the first variation of the relative entropy is
$\frac{\delta}{\delta m}\KL(m\|\nu)=\log\frac{m}{\nu}+\mathbf 1$,
and $\mathsf K_{g_t}\mathbf 1=0$, equation \eqref{eq:dotmt} can be
written in the Onsager form
\begin{equation}
\label{eq:onsager-form}
\dot m_t=-\mathsf K_{g_t}\frac{\delta}{\delta m}\KL(m_t\|\nu).
\end{equation}
Thus the thermodynamic force is $-\frac{\delta}{\delta m}\KL(m\|\nu)$,
and the bridge-induced mobility $\mathsf K_{g_t}$ converts this force into
the marginal velocity $\dot m_t$.

\begin{proposition}
\label{prop:marginal-dissipation}
Let $g_t$ solve \eqref{eq:SBOF} with $g_0\in\mathcal H$, and set
$m_t=m_{g_t}$. Then, for all $t\ge0$,
\begin{equation}
\label{eq:marginal-dissipation}
\frac{d}{dt}\KL(m_t\,\|\,\nu)
=-\left\langle\log\frac{m_t}{\nu},\mathsf K_{g_t}\log\frac{m_t}{\nu}\right\rangle
=-\mathcal E_{g_t}\!\left(\log\frac{m_t}{\nu}\right).
\end{equation}
\end{proposition}

\begin{proof}
Since $m_t$ is a probability distribution for all $t$, we have
$\langle\mathbf 1,\dot m_t\rangle=0$. Therefore
\[
\frac{d}{dt}\KL(m_t\|\nu)
=\left\langle\log\frac{m_t}{\nu}+\mathbf 1,\dot m_t\right\rangle
=\left\langle\log\frac{m_t}{\nu},\dot m_t\right\rangle.
\]
Using \eqref{eq:dotmt}, we obtain
\[
\frac{d}{dt}\KL(m_t\|\nu)
=-\left\langle\log\frac{m_t}{\nu},\mathsf K_{g_t}\log\frac{m_t}{\nu}\right\rangle.
\]
The final equality follows from the definition
$\mathcal E_g(f)=\langle f,\mathsf K_g f\rangle$.
\end{proof}

The quantity
\[
\mathcal I_{g_t}(m_t\|\nu):=\mathcal E_{g_t}\!\left(\log\frac{m_t}{\nu}\right)
\]
will be referred to as the nonlocal Fisher information associated with the
bridge-induced Onsager geometry. Identity~\eqref{eq:marginal-dissipation}
plays the same structural role as the entropy-production identity in
discrete gradient-flow geometries for Markov chains
\cite{maas2011gradient,mielke2011geodesic,erbar2012ricci}. It therefore
suggests a bridge-induced analogue of curvature or Bochner theory for the
nonlocal forms $\mathcal E_g$. We do not pursue this direction here and
leave its quantitative development for future work.

\section{Uniform Functional Inequalities and Exponential Relaxation}
\label{sec:exponential}

The entropy identity \eqref{eq:marginal-dissipation} expresses dissipation
through the bridge-induced Dirichlet form $\mathcal E_{g_t}$. Because the
Onsager operator $\mathsf K_{g_t}$ evolves with the potential trajectory,
the coercivity of $\mathcal E_{g_t}$ is not fixed. To obtain a quantitative
convergence rate for the SBOF, we establish a uniform Poincar\'e inequality
for the family $\{\mathcal E_g\}_{g\in\mathcal C_{g_0}}$ on the compact
semi-dual sublevel set determined by the initial condition, and combine it
with an entropy--variance comparison for $\log(m_g/\nu)$.

\subsection{Uniform control of the Onsager geometry}
\label{subsec:uniform-control}

Fix $g_0\in\mathcal H$ and define the semi-dual sublevel set
\begin{equation}
\label{eq:sublevel-set}
\mathcal C_{g_0}
=
\left\{g\in\mathcal H:\mathcal J(g)\le \mathcal J(g_0)\right\}.
\end{equation}
By the Lyapunov identity \eqref{eq:J-dissipation}, every solution of the
potential flow \eqref{eq:SBOF} satisfies $g_t\in\mathcal C_{g_0}$ for all
$t\ge0$. Since $\mathcal J$ is coercive on $\mathcal H$, the set
$\mathcal C_{g_0}$ is compact. The maps $g\mapsto p_g$, $g\mapsto m_g$,
and $g\mapsto\mathsf K_g$ are continuous on this compact set; hence the
family of nonlocal Dirichlet forms $\mathcal E_g$ is uniformly controlled
along the trajectory.

We first establish uniform coercivity.

\begin{theorem}
\label{thm:uniform-poincare}
Let $g_0\in\mathcal H$. There exists $\lambda_{g_0}>0$ such that, for all
$g\in\mathcal C_{g_0}$ and all $f:S\to\mathbb R$ with
$\mathbb E_\nu[f]=0$,
\begin{equation}
\label{eq:PI}
\langle f,\mathsf K_g f\rangle
\ge
\lambda_{g_0}\|f\|_{L^2(\nu)}^2.
\end{equation}
\end{theorem}

\begin{proof}
For fixed $g$, the covariance representation gives
\[
\langle f,\mathsf K_g f\rangle
=
\sum_{x\in S}\mu(x)\operatorname{Var}_{p_g(\cdot|x)}(f).
\]
Since $K(x,y)>0$, every $p_g(\cdot|x)$ has full support on $S$. Hence
$\langle f,\mathsf K_g f\rangle=0$ if and only if $f$ is constant; if
$\mathbb E_\nu[f]=0$, then $f=0$. Thus $\mathsf K_g$ is strictly positive
on the $\nu$-mean-zero subspace. Define
\[
\lambda(g)
=
\inf\left\{\langle f,\mathsf K_g f\rangle:
\mathbb E_\nu[f]=0,\;
\|f\|_{L^2(\nu)}=1\right\}.
\]
The map $g\mapsto\lambda(g)$ is continuous, as it is the smallest
eigenvalue of the restriction of $\mathsf K_g$ to a fixed subspace.
Since $\mathcal C_{g_0}$ is compact,
\begin{equation}\label{lambdag0}\lambda_{g_0}:=\min_{g\in\mathcal C_{g_0}}\lambda(g)>0.
\end{equation}
\end{proof}

Next, we compare relative entropy with the $L^2(\nu)$-fluctuation of the
logarithmic density ratio.

\begin{lemma} 
\label{lem:entropy-variance}
Let $g_0\in\mathcal H$. There exists $C_{g_0}>0$ such that, for every
$g\in\mathcal C_{g_0}$,
\begin{equation}
\label{eq:entropy-variance}
\KL(m_g\,\|\,\nu)
\le
C_{g_0}
\left\|
\log\frac{m_g}{\nu}
-
\mathbb E_\nu\!\left[
\log\frac{m_g}{\nu}
\right]
\right\|_{L^2(\nu)}^2.
\end{equation}
\end{lemma}

\begin{proof}
Since $g\mapsto m_g$ is continuous and $\mathcal C_{g_0}$ is compact,
the set $\{m_g:g\in\mathcal C_{g_0}\}$ is a compact subset of the interior
of the probability simplex. Hence there exist constants
$0<a_{g_0}\le b_{g_0}<\infty$ such that
$a_{g_0}\le m_g(y)/\nu(y)\le b_{g_0}$ for all $g\in\mathcal C_{g_0}$ and
$y\in S$. Set $r_g=m_g/\nu$, $\ell_g=\log r_g$, and
$\bar\ell_g=\mathbb E_\nu[\ell_g]$. Since $\mathbb E_\nu[r_g]=1$,
\[
\KL(m_g\|\nu)=\mathbb E_\nu[r_g\log r_g]
=\mathbb E_\nu[\Phi(r_g)],
\]
where $\Phi(r)=r\log r-r+1$. On $[a_{g_0},b_{g_0}]$, we have $\Phi(r)\le C_1(r-1)^2$ for some
$C_1>0$, and $|r_g(y)-r_g(z)|\le b_{g_0}|\ell_g(y)-\ell_g(z)|$. Thus
\[
\KL(m_g\|\nu)\le C_1\|r_g-1\|_{L^2(\nu)}^2
=C_1\operatorname{Var}_\nu(r_g)
\le C_1 b_{g_0}^2\operatorname{Var}_\nu(\ell_g),
\]
which gives \eqref{eq:entropy-variance} with
$C_{g_0}=C_1 b_{g_0}^2$.
\end{proof}

\subsection{Exponential convergence}
\label{subsec:exponential-convergence}

We now combine the marginal entropy-dissipation identity
\eqref{eq:marginal-dissipation}, the uniform Poincar\'e inequality
\eqref{eq:PI}, and the entropy--variance comparison
\eqref{eq:entropy-variance}.

\begin{theorem}
\label{thm:exponential-convergence}
Let $g_t$ solve the potential flow \eqref{eq:SBOF} with $g_0\in\mathcal H$,
and set $m_t=m_{g_t}$. Then
\begin{equation}
\label{eq:exp-convergence}
\KL(m_t\|\nu)
\le
\KL(m_0\|\nu)
\exp(-\omega_{g_0}t),
\qquad t\ge0,
\end{equation}
where $\omega_{g_0}=\lambda_{g_0}/C_{g_0}$, with $\lambda_{g_0}$ from
Theorem~\ref{thm:uniform-poincare} and $C_{g_0}$ from
Lemma~\ref{lem:entropy-variance}.
\end{theorem}

\begin{proof}
Set $h_t=\log(m_t/\nu)$. By Proposition~\ref{prop:marginal-dissipation},
\[
\frac{d}{dt}\KL(m_t\|\nu)=-\langle h_t,\mathsf K_{g_t}h_t\rangle.
\]
Since $\mathsf K_{g_t}\mathbf 1=0$ and $\mathsf K_{g_t}$ is symmetric, we may subtract the $\nu$-mean of $h_t$:
\[
\langle h_t,\mathsf K_{g_t}h_t\rangle
=\left\langle h_t-\mathbb E_\nu[h_t],
\mathsf K_{g_t}\bigl(h_t-\mathbb E_\nu[h_t]\bigr)\right\rangle.
\]
Because $g_t\in\mathcal C_{g_0}$, Theorem~\ref{thm:uniform-poincare} gives
\[
\langle h_t,\mathsf K_{g_t}h_t\rangle
\ge\lambda_{g_0}\|h_t-\mathbb E_\nu[h_t]\|_{L^2(\nu)}^2,
\]
and Lemma~\ref{lem:entropy-variance} yields
\[
\|h_t-\mathbb E_\nu[h_t]\|_{L^2(\nu)}^2
\ge\frac{1}{C_{g_0}}\KL(m_t\|\nu).
\]
Therefore
\[
\frac{d}{dt}\KL(m_t\|\nu)
\le-\frac{\lambda_{g_0}}{C_{g_0}}\KL(m_t\|\nu),
\]
and Gr\"onwall's inequality gives \eqref{eq:exp-convergence}.
\end{proof}

\section{Dynamic Reconstruction and Numerical Illustration}
\label{sec:dynamic-reconstruction}

The previous sections identify the gauge-fixed Schr\"odinger potential
$g^\star$, the induced SBOF governed by the Onsager operator
$\mathsf K_{g_t}$, and the associated nonlocal Dirichlet geometry. We now
reconstruct the full path-space dynamics from the terminal potential through
the space-time Doob transform \cite{ doob1957conditional, jamison1975markov,leonard2014survey}. This reconstruction is motivated by the recent surge of interest in
SB as theoretically grounded generative models, which
has established SB as a principled variational framework for diffusion-type
generation \cite{debortoli2021diffusion,peluchetti2023diffusion, shi2023diffusion,liu2023introduction, tang2026foundations, korotin2024light,gushchin2024light}. In discrete state
spaces, however, the geometry of probability dynamics is subtle: different
choices of graph, generator, mobility, or curvature lead to distinct
dissipation structures
\cite{maas2011gradient,mielke2011geodesic,erbar2012ricci,
ollivier2009ricci,fathi2018curvature}. The Doob transform shows how the
limiting potential $g^\star$ realizes the dynamic SB whose marginal
relaxation is encoded by the Onsager geometry identified above.

\subsection{Doob transform and bridge reconstruction}
\label{subsec:doob-reconstruction}

Let \(g:S\to\mathbb R\). Define the space-time harmonic function
\begin{equation}
\label{eq:harmonic-function}
\varphi_s^g=e^{(1-s)Q}e^g,\qquad s\in[0,1].
\end{equation}
Since \(K_t=e^{tQ}\) is a Markov semigroup and \(e^g>0\), the function
\(\varphi_s^g\) defined in \eqref{eq:harmonic-function} satisfies
\(\varphi_s^g(x)>0\) for all \(s\in[0,1]\) and \(x\in S\).
The corresponding Doob-transformed generator is defined, for \(x\neq y\), by
\begin{equation}
\label{eq:doob-generator}
Q_s^g(x,y)=Q(x,y)\frac{\varphi_s^g(y)}{\varphi_s^g(x)},
\end{equation}
with diagonal entries chosen so that rows sum to zero. The generator
\eqref{eq:doob-generator} is the classical space-time Doob \(h\)-transform
of the reference Markov chain
\cite{doob1957conditional,jamison1975markov,leonard2014survey}.

The time-inhomogeneous Markov chain generated by \eqref{eq:doob-generator},
initialized from \(X_0\sim\mu\), has the same conditional bridges as the
reference process and endpoint coupling \(\pi_g\). Equivalently, its path
measure is
\begin{align}
\label{eq:Pg-path}
\mathbf P^g(d\omega)
=
\sum_{x,y\in S}
\pi_g(x,y)\,
\mathbf R_\mu(d\omega\mid X_0=x,X_1=y).
\end{align}

In particular, when \(g=g^\star\), the endpoint marginal satisfies
\(m_{g^\star}=\nu\), and the path measure \eqref{eq:Pg-path} with
\(g=g^\star\) is the SB path measure. Thus
Theorem~\ref{thm:SBOF-convergence} and
Corollary~\ref{cor:coupling-path-convergence} together yield convergence
of the reconstructed dynamic bridge.

\subsection{Euler discretization}
\label{subsec:euler-discretization}

For the numerical illustration below, we approximate the potential flow
\eqref{eq:SBOF} by an explicit Euler scheme on a finite time interval
\([0,T]\), with \(T>0\) chosen to capture the long-time relaxation regime.
Let \(N_{\rm it}\in\mathbb N\) be the number of time steps and set
\begin{equation}
\label{eq:time-grid}
\Delta t=\frac{T}{N_{\rm it}},\qquad t_k=k\Delta t,\qquad k=0,\ldots,N_{\rm it}.
\end{equation}
Starting from \(g_0\in\mathcal H\), we define \(g_k\approx g_{t_k}\) via the
grid \eqref{eq:time-grid}. Given \(g_k\), compute the induced marginal
\(m_{g_k}\) and set
\begin{equation}
\label{eq:hk-def}
h_k=\log\frac{m_{g_k}}{\nu}.
\end{equation}
The explicit Euler update is
\begin{equation}
\label{eq:euler-sbof}
g_{k+1}
=
g_k
-
\Delta t
\left(
h_k
-
\mathbb E_\nu[h_k]\mathbf 1
\right),
\qquad
k=0,\ldots,N_{\rm it}-1,
\end{equation}
where \(h_k\) is defined in \eqref{eq:hk-def}. The subtraction of the
\(\nu\)-mean preserves the gauge in exact arithmetic: if
\(\mathbb E_\nu[g_k]=0\), then \(\mathbb E_\nu[g_{k+1}]=0\). In
floating-point computations, we additionally apply the projection
\begin{equation}
\label{eq:gauge-projection}
g_{k+1}
\leftarrow
g_{k+1}
-
\mathbb E_\nu[g_{k+1}]\mathbf 1
\end{equation}
after each step, in order to remove numerical drift in the gauge direction.
The projection \eqref{eq:gauge-projection} merely removes numerical drift
and leaves both the coupling \(\pi_{g_{k+1}}\) and the marginal
\(m_{g_{k+1}}\) unchanged.

At each time step we also compute the Onsager entropy-production diagnostic
\begin{equation}
\label{eq:discrete-fisher-diagnostic}
\mathcal I_{g_k}(m_{g_k}\|\nu)
=
\mathcal E_{g_k}\!\left(\log\frac{m_{g_k}}{\nu}\right)
=
\sum_{x\in S}
\mu(x)
\operatorname{Var}_{p_{g_k}(\cdot|x)}
\left(
\log\frac{m_{g_k}}{\nu}
\right),
\end{equation}
which follows from the covariance representation of \(\mathsf K_{g_k}\) and
can be evaluated without explicitly forming the full matrix
\(\mathsf K_{g_k}\). The diagnostic \eqref{eq:discrete-fisher-diagnostic}
measures the instantaneous entropy production along the discretized
trajectory.

\subsection{Rare-state numerical illustration}
\label{subsec:rare-state-illustration}

We illustrate the potential flow \eqref{eq:SBOF} on a challenging
finite-state bridge problem where the target law $\nu$ exhibits rare
modes with very small mass. The state space is a $30\times 30$ grid,
so that $N=900$. The reference generator $Q$ is a local continuous-time
random-walk generator on the grid, and $K=e^Q$ is computed explicitly.
The source law $\mu$ is a simple localized distribution, whereas the
target law $\nu$ is an eight-mode distribution with two low-mass modes.
This construction tests whether the logarithmic force controls
multiplicative errors $\log(m_g/\nu)$ in rare regions, rather than only
global absolute discrepancies.

We compute a high-accuracy reference solution using classical
Sinkhorn/IPF iterations \cite{peyre2019computational}. We compare the
terminal marginal $m_{g_T}$, the endpoint coupling $\pi_{g_T}$, and
the logarithmic relative error against this reference. In addition to
$\KL(m_{g_T}\|\nu)$ and total variation distance, we report
\begin{equation}
\label{eq:max-log-error}
\max_{y\in S}\left|\log\frac{m_{g_T}(y)}{\nu(y)}\right|,
\end{equation}
which measures the worst multiplicative terminal error. This diagnostic
is particularly relevant when $\nu(y)$ is very small, since small
absolute errors can correspond to large relative errors in rare states.
The rare-state max log-error is computed over the two lowest-mass modes
of $\nu$.

{\small
\begin{table}[t]
\centering
\small
\setlength{\tabcolsep}{4pt}
\renewcommand{\arraystretch}{1.18}
\begin{tabular}{|l|c|c|c|c|}
\hline
Method
& $\KL(m_T\|\nu)$
& $\TV(m_T,\nu)$
& Max log-error
& Rare max log-error \\
\hline
Sinkhorn/IPF
& $9.82{\times}10^{-15}$
& $6.64{\times}10^{-8}$
& $2.26{\times}10^{-7}$
& $1.14{\times}10^{-7}$ \\

\textbf{Potential flow}
& $\mathbf{<<10^{-15}}$
& $\mathbf{1.35{\times}10^{-12}}$
& $\mathbf{4.60{\times}10^{-12}}$
& $\mathbf{2.32{\times}10^{-12}}$ \\

\textbf{Euclidean residual}
& $6.13{\times}10^{-1}$
& $2.71{\times}10^{-1}$
& $2.07{\times}10^{1}$
& $9.35{\times}10^{0}$ \\
\hline
\end{tabular}
\caption{Eight-mode rare-state bridge, averaged over ten random seeds.
Two target modes carry very small probability mass. The potential flow
reaches reference-level terminal accuracy. The Euclidean residual update
fails to recover the target marginal and leaves large rare-state
log-errors.}
\end{table}
}

\begin{figure}[t]
    \centering
    \includegraphics[width=0.9\linewidth]{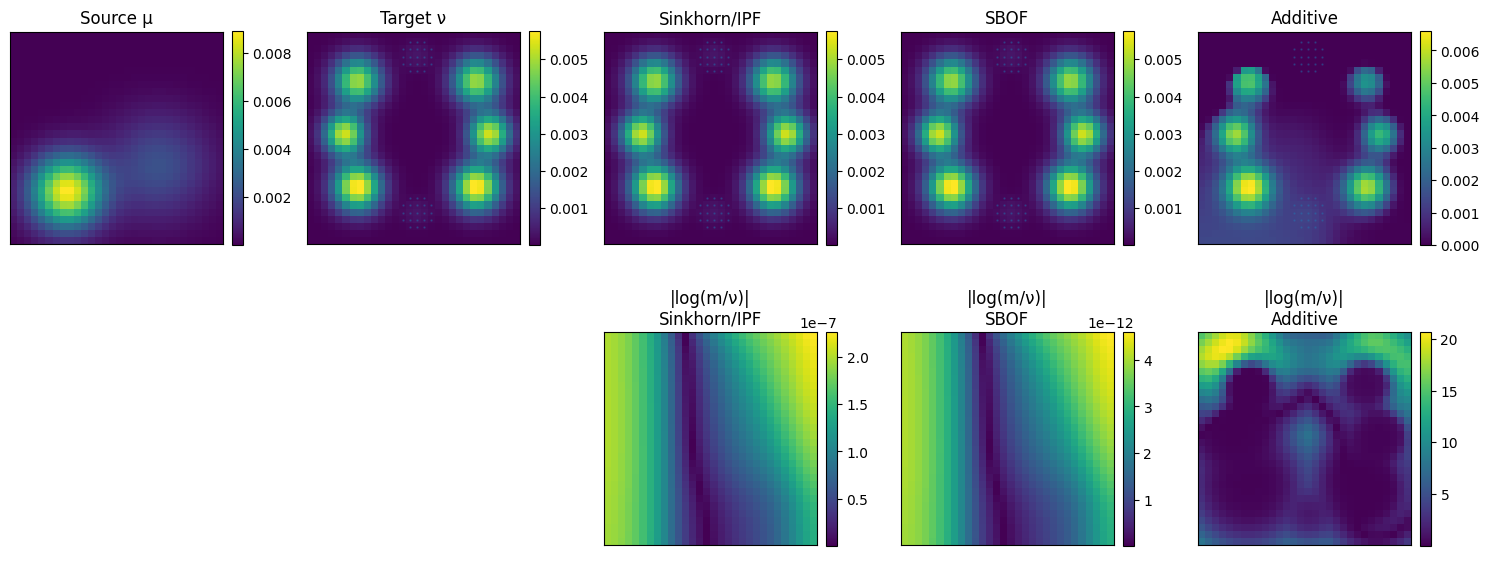}    \caption{Rare-state multimodal bridge. Top: target and learned terminal
    marginals. Bottom: relative log-error maps
    $|\log(m_T/\nu)|$. The log-error diagnostic \eqref{eq:max-log-error}
    reveals errors in rare regions that are nearly invisible in absolute mass.}
    \label{fig:rare-state-heatmaps}
\end{figure}

The results illustrate the role of the logarithmic force in controlling
relative terminal errors. While a Euclidean residual update can reduce
some absolute discrepancies, it may leave large multiplicative errors
in rare regions. By contrast, the potential flow acts directly on
$\log(m_g/\nu)$, which is the entropy force appearing in the Onsager
form \eqref{eq:onsager-form} of the SBOF. The numerical behavior is
therefore consistent with the entropy-dissipation identity
\eqref{eq:marginal-dissipation}, where the Dirichlet form
$\mathcal{E}_{g_t}$ measures the instantaneous entropy production
along the trajectory.

\section{Conclusion and Perspectives}
\label{sec:conclusion}

We have shown that the finite-state SB induces an Onsager geometry through
the semi-dual marginal map. The operator $\mathsf K_g=Dm_g=\nabla^2\mathcal J(g)$
acts as the mobility in the terminal marginal equation, and its Dirichlet form
is the entropy-production functional along the induced marginal flow. This
identifies a bridge-dependent nonlocal geometry generated by the Schr\"odinger
parametrization itself, not prescribed by any external structure.

This theoretical analysis arrives as discrete diffusion models are being
re-examined through geometric mechanics. Recent work has translated the
metric $W_K$ of \cite{maas2011gradient} and the JKO proximal scheme into
practical machine learning methodologies \cite{rancati2026learning},
demonstrating that the geometric ``machine'' underlying discrete diffusion
is becoming computationally tractable. The SBOF framework contributes to
this emerging picture by identifying the bridge-induced Onsager operator
$\mathsf K_g$ as the natural, state-dependent mobility for
entropy-dissipating transport, together with the rigorous dissipative
structure that any geometry-aware discrete generative model must satisfy.
We expect the interplay between such operator-theoretic foundations and
algorithmic advances in discrete gradient-flow learning to produce, in the
coming years, diffusion models that are simultaneously more theoretically
grounded and geometrically interpretable.

Several directions remain open. First, the exponential rate obtained here
rests on compactness of semi-dual sublevel sets. A quantitative theory
should relate the spectrum of $\mathsf K_g$ to the spectral gap of the
semigroup $K_t=e^{tQ}$ and to the laws $\mu,\nu$, turning the abstract
Poincar\'e constant into an explicit conditioning parameter. Such
estimates would be essential for coupling the SBOF with the numerical
geodesic solvers recently developed in the discrete JKO literature
\cite{rancati2026learning}.

Second, the nonlocal forms $\mathcal E_g(f)=\langle f,\mathsf K_g f\rangle$
suggest a bridge-induced analogue of discrete curvature theory. Existing
finite-state geometries start from a prescribed Markov generator and study
entropy convexity or Bochner inequalities. Here the mobility varies with
the Schr\"odinger potential. Understanding whether the forms
$\mathcal E_g$ admit curvature or Bochner-type estimates, and how they
depend on $Q,\mu,\nu$, is a natural next step; a positive answer would
yield quantitative acceleration criteria for the SBOF.

Third, strict positivity of $K,\mu,\nu$ ensures smoothness and coercivity.
Many applications involve sparse kernels, forbidden states, or structural
zeros. Extending the theory to such constrained supports would require a
support-adapted geometry, particularly for sequence-generation on Hamming
graphs where transitions are restricted to local token mutations.

Finally, the Doob-transform reconstruction suggests applications to
conditional and guided generation on finite state spaces
\cite{nisonoff2025unlocking,schiff2025simple,uehara2025reward,wallace2024diffusion}.
Perturbing the terminal potential by a reward $R$ changes the marginal
according to the linear response
$m_{g+\varepsilon R}=m_g+\varepsilon \mathsf K_g R+O(\varepsilon^2)$.
Thus $\mathsf K_g$ is the local response operator mapping potential
perturbations to marginal changes. This perspective may lead to
geometry-aware conditioning or fine-tuning procedures for discrete
generative models, especially in sequence-generation problems where rare
states and relative errors are important.

\bibliographystyle{siamplain}
\bibliography{references}

\appendix

\section{Dynamic-to-Static Reduction and Semi-Duality}

The proof is given in the supplementary material. We recall only the main structural points,
since they will be used later. The Hessian of \(\mathcal J\) is the
differential of the marginal map \(g\mapsto m_g\):
\begin{equation*}
\nabla^2\mathcal J(g)=Dm_g.
\end{equation*}
Moreover, for every \(f:S\to\mathbb R\),
\begin{equation*}
\label{eq:hessian-cov-preview1}
\left\langle f,\nabla^2\mathcal J(g)f\right\rangle
=
\sum_{x\in S}
\mu(x)\operatorname{Var}_{p_g(\cdot|x)}(f).
\end{equation*}
Thus \(\mathcal J\) is convex, and the positivity of \(K\) implies strict
convexity modulo additive constants. After restricting to the gauge-fixed
space \(\mathcal H\), coercivity and strict convexity yield a unique
minimizer \(g^\star\). The gradient identity \eqref{eq:grad-J} then gives
the Schrödinger marginal equation
\(
m_{g^\star}=\nu .
\)

\section{Static Theory on the Gauge-Fixed Space}
\label{app:static-theory}

We prove Theorem~\ref{thm:exist-unique}. The argument is based on three
facts: the semi-dual functional is coercive on the gauge-fixed space
\(\mathcal H\), it is strictly convex modulo additive constants, and its
first-order optimality condition is exactly the Schrödinger marginal
equation. The invariance under additive constants follows directly from
\(Z_{g+c\mathbf 1}=e^cZ_g\); the elementary verification is given in the
supplementary material.

\subsection{Coercivity}
\label{app:coercivity}

We first prove that \(\mathcal J\) is coercive on the gauge-fixed subspace
\[
\mathcal H
=
\left\{
g\in\mathbb R^S:\mathbb E_\nu[g]=0
\right\}.
\]
Let \((g_n)_{n\ge1}\subset\mathcal H\) be such that
\(\|g_n\|\to\infty.
\)
Write
\[
g_n=r_nu_n,
\qquad
r_n=\|g_n\|\to\infty,
\qquad
\|u_n\|=1,
\]
where $u_n \in \mathcal H$. Since \(S\) is finite, after extracting a subsequence we may assume that
\(u_n\to u\) in \(\mathbb R^S\). Since \(\mathcal H\) is closed, we have
\(u\in\mathcal H\), and since \(\|u_n\|=1\), we also have \(\|u\|=1\).

For each \(x\in S\),
\[
\log\left(
\sum_{y\in S}K(x,y)e^{r_nu_n(y)}
\right)
=
r_n\max_{y\in S}u_n(y)+O(1),
\]
where the \(O(1)\) term is uniform in \(x\). Indeed, because \(S\) is finite
and \(K(x,y)>0\), there exist constants \(0<k_- \le k_+<\infty\) such that
\[
k_-\le K(x,y)\le k_+,
\qquad x,y\in S.
\]
Thus the logarithmic sum is controlled, up to an additive constant
independent of \(n\) and \(x\), by its largest exponential term. Using \(g_n=r_nu_n\), we obtain
\begin{align*}
\frac{\mathcal J(g_n)}{r_n}
&=
\sum_{x\in S}\mu(x)
\frac{1}{r_n}
\log\left(
\sum_{y\in S}K(x,y)e^{r_nu_n(y)}
\right)
-
\sum_{y\in S}\nu(y)u_n(y) \\
&=
\max_{y\in S}u_n(y)
-
\mathbb E_\nu[u_n]
+
o(1).
\end{align*}
Since \(u_n\in\mathcal H\), we have \(\mathbb E_\nu[u_n]=0\). Passing to the
limit gives
\[
\frac{\mathcal J(g_n)}{r_n}
\to
\max_{y\in S}u(y).
\]
Now \(u\in\mathcal H\), \(\|u\|=1\), and \(\nu(y)>0\) for every \(y\in S\).
Hence \(u\) is not identically zero. Moreover, \(\max_y u(y)\) must be
strictly positive. Indeed, if \(\max_y u(y)\le0\), then \(u(y)\le0\) for all
\(y\), and the condition \(\mathbb E_\nu[u]=0\) with \(\nu>0\) would force
\(u=0\), contradicting \(\|u\|=1\). Therefore
\[
\max_{y\in S}u(y)>0.
\]
It follows that \(\mathcal J(g_n)\to+\infty\). Thus \(\mathcal J\) is
coercive on \(\mathcal H\).

\subsection{Strict convexity}
\label{app:strict-convexity}
We next prove strict convexity modulo constants. By
Proposition~\ref{prop:gradient-J},
\(
\nabla\mathcal J(g)=m_g-\nu.
\)
Hence
\(
\nabla^2\mathcal J(g)=Dm_g.
\)
We compute the differential of the marginal map \(g\mapsto m_g\). Let
\(h:S\to\mathbb R\). Differentiating
\[
p_g(y|x)=\frac{K(x,y)e^{g(y)}}{Z_g(x)}
\]
in the direction \(h\) gives
\[
Dp_g(y|x)[h]
=
p_g(y|x)
\left(
h(y)-\sum_{z\in S}p_g(z|x)h(z)
\right).
\]
Since
\[
m_g(y)=\sum_{x\in S}\mu(x)p_g(y|x),
\]
we obtain
\[
Dm_g(y)[h]
=
\sum_{x\in S}\mu(x)p_g(y|x)
\left(
h(y)-\sum_{z\in S}p_g(z|x)h(z)
\right).
\]
Consequently,
\begin{align*}
\left\langle h,\nabla^2\mathcal J(g)h\right\rangle
&=
\sum_{y\in S}h(y)Dm_g(y)[h],\\
&=
\sum_{x\in S}\mu(x)
\left[
\sum_{y\in S}p_g(y|x)h(y)^2
-
\left(
\sum_{y\in S}p_g(y|x)h(y)
\right)^2
\right], \\
&=
\sum_{x\in S}\mu(x)
\operatorname{Var}_{p_g(\cdot|x)}(h).
\end{align*}
This proves that \(\nabla^2\mathcal J(g)\) is positive semidefinite.
Moreover, since \(K(x,y)>0\), each probability vector \(p_g(\cdot|x)\) has
full support on \(S\). Hence
\(
\operatorname{Var}_{p_g(\cdot|x)}(h)=0
\)
if and only if \(h\) is constant on \(S\). Therefore
\(
\left\langle h,\nabla^2\mathcal J(g)h\right\rangle=0
\)
if and only if
\(
h\in\operatorname{span}\{\mathbf 1\}.
\)
Equivalently,
\[
\operatorname{Ker}\nabla^2\mathcal J(g)
=
\operatorname{span}\{\mathbf 1\}.
\]
Thus \(\mathcal J\) is strictly convex on the quotient
\(\mathbb R^S/\mathbb R\mathbf 1\), and therefore strictly convex on the
gauge-fixed hyperplane \(\mathcal H\).

\subsection{First-order optimality and uniqueness}
\label{app:first-order-optimality}

Since \(\mathcal J\) is continuous and coercive on the finite-dimensional
space \(\mathcal H\), it attains a minimizer \(g^\star\in\mathcal H\). The
strict convexity on \(\mathcal H\), proved above, implies that this minimizer
is unique.

It remains to identify the first-order optimality condition. Since
\(g^\star\) minimizes \(\mathcal J\) over \(\mathcal H\),
\[
\left\langle \nabla\mathcal J(g^\star),h\right\rangle=0,
\qquad
h\in\mathcal H.
\]
By Proposition~\ref{prop:gradient-J},
\(
\nabla\mathcal J(g^\star)=m_{g^\star}-\nu.
\)
Moreover, \(m_{g^\star}\) and \(\nu\) are probability measures, so
\(
\left\langle m_{g^\star}-\nu,\mathbf 1\right\rangle=0.
\)
Thus \(m_{g^\star}-\nu\) is orthogonal to \(\mathcal H\) and also to
\(\mathbf 1\). Since
\(
\mathbb R^S
=
\mathcal H\oplus \operatorname{span}\{\mathbf 1\},
\)
we conclude that
\(
m_{g^\star}-\nu=0.
\)
Hence
\(
m_{g^\star}=\nu.
\)
Therefore \(g^\star\) is the unique gauge-fixed Schrödinger potential. By
Theorem~\ref{thm:finite-duality}, the coupling
\[
\pi_{g^\star}(x,y)=\mu(x)p_{g^\star}(y|x)
\]
solves the static SBP. Finally, the entropy objective
is strictly convex in \(\pi\) on \(\Pi(\mu,\nu)\), and hence this coupling is
the unique solution of the static problem.

\section{Well-Posedness and Convergence of the SBOF}
\label{app:sbof-convergence}

We prove Theorem~\ref{thm:SBOF-convergence}. Throughout this section,
\(g_0 \in\mathcal H\) is fixed, and \(g_t\) denotes the solution of
potential flow \eqref{eq:SBOF} whenever it exists. The proof proceeds by establishing local
well-posedness, gauge preservation, the Lyapunov identity, global
existence, and convergence to the unique gauge-fixed Schrödinger potential.

\subsection{Local well-posedness and gauge preservation}
\label{app:local-wellposedness-gauge}

We prove the local well-posedness of the potential flow \eqref{eq:SBOF} and the
preservation of the gauge constraint. Define the vector field
\begin{equation}
\label{eq:F-g}
F(g)
=
-\log\frac{m_g}{\nu}
+
\mathbb E_\nu\!\left[
\log\frac{m_g}{\nu}
\right]\mathbf 1,
\qquad g\in\mathbb R^S.
\end{equation}

Because each term in \eqref{eq:mg} is strictly positive
(\(\mu(x)>0\), \(K(x,y)>0\), \(e^{g(y)}>0\), \(Z_g(x)>0\)),
we obtain \(m_g(y) > 0\) for every \(g\in\mathbb R^S\) and \(y\in S\). Moreover, the map \(g\mapsto m_g\) is \(C^\infty\) on \(\mathbb R^S\), since
it is obtained from finite sums, exponentials, products, and divisions by
the strictly positive quantities \(Z_g(x)\). Hence
\(g\mapsto \log(m_g/\nu)\) is \(C^\infty\), and therefore \(F\) defined in
\eqref{eq:F-g} is \(C^\infty\) on \(\mathbb R^S\).

By the Cauchy--Lipschitz theorem, for every $g_0\in\mathbb R^S$ the initial-value problem
\begin{equation}
\label{eq:ode-local}
\dot g_t=F(g_t),
\end{equation}
 with $g_{t=0}=g_0$ admits a unique maximal solution on an interval
$[0,T_{\max})$, where $T_{\max}\in(0,\infty]$. It remains to check that the gauge is preserved. Along any solution
of \eqref{eq:ode-local},
\begin{equation}
\label{eq:gauge-preservation}
\frac{d}{dt}\mathbb E_\nu[g_t]
=
\mathbb E_\nu[\dot g_t]
=
-\mathbb E_\nu\!\left[
\log\frac{m_{g_t}}{\nu}
\right]
+
\mathbb E_\nu\!\left[
\log\frac{m_{g_t}}{\nu}
\right]
=
0.
\end{equation}
Thus, by \eqref{eq:gauge-preservation},
\(\mathbb E_\nu[g_t]=\mathbb E_\nu[g_0]\) for all \(t \in [0,T_{\max})\).
In particular, if \(g_0\in\mathcal H\), then
\(g_t\in\mathcal H\) for all \(t\in [0, T_{\max})\).

\subsection{Dissipation identity}
\label{app:dissipation-identity}

We prove the Lyapunov identity for the semi-dual functional. By
Proposition~\ref{prop:gradient-J},
\(
\nabla\mathcal J(g)=m_g-\nu.
\)
Therefore, along a smooth solution of the potential flow \eqref{eq:SBOF},
\begin{align}\label{dJ_t}
\frac{d}{dt}\mathcal J(g_t)
=
\left\langle
\nabla\mathcal J(g_t),\dot g_t
\right\rangle
=
\left\langle
m_{g_t}-\nu,\dot g_t
\right\rangle .
\end{align}
Using the definition of the potential flow  \eqref{eq:SBOF}, we
obtain
\begin{align}\label{mgt}
\left\langle
m_{g_t}-\nu,\dot g_t
\right\rangle
=
-
\left\langle
m_{g_t}-\nu,
\log\frac{m_{g_t}}{\nu}
\right\rangle
+
\mathbb E_\nu\!\left[
\log\frac{m_{g_t}}{\nu}
\right]
\left\langle m_{g_t}-\nu,\mathbf 1\right\rangle .
\end{align}
The gauge correction vanishes, since
\begin{align}\label{gf}
\left\langle m_{g_t}-\nu,\mathbf 1\right\rangle
=
\sum_{y\in S}m_{g_t}(y)-\sum_{y\in S}\nu(y)
=
1-1
=
0.
\end{align}
By combining \eqref{mgt} and \eqref{gf}, equation \eqref{dJ_t} can be rewritten as
\begin{align}
\frac{d}{dt}\mathcal J(g_t)
=&
-
\left\langle
m_{g_t}-\nu,
\log\frac{m_{g_t}}{\nu}
\right\rangle,\notag\\
 =&-\KL(m_{g_t}\|\nu)
-\KL(\nu\|m_{g_t})
\le0.\label{dissipation}
\end{align}
This proves the dissipation identity.

\subsection{Global existence}
\label{app:global-existence}

We now prove that the maximal solution \((g_t)_{t\in[0,T_{\max})}\) of 
potential flow \eqref{eq:SBOF} is global, that is,
\(
T_{\max}=+\infty\). Since the gauge is
preserved, we have
\( g_t\in\mathcal H\) for all $t\in [0,T_{\max})$.
Moreover, by the dissipation identity (\ref{dissipation}),
\begin{align*}
\mathcal J(g_t)\le \mathcal J(g_0),
\qquad
\forall t \in [0,T_{\max}).
\end{align*}
Consequently, for all \(t\in[0,T_{\max})\), the trajectory remains in the
sublevel set
\begin{align*}
g_t\in\mathcal C_{g_0}
:=
\left\{
g\in\mathcal H:
\mathcal J(g)\le \mathcal J(g_0)
\right\}.
\end{align*}
By the coercivity of \(\mathcal J\) on \(\mathcal H\), proved in
Appendix~\ref{app:static-theory}, the sublevel set
\(\mathcal C_{g_0}\) is compact.
Since the vector field $
F(g)$ defined by \eqref{eq:F-g}
is smooth on \(\mathbb R^S\), hence bounded on the compact set
\(\mathcal C_{g_0}\). Therefore the solution cannot leave every compact subset
of \(\mathbb R^S\) as \(t\uparrow T_{\max}\). By the standard continuation
criterion for finite-dimensional ordinary differential equations, this
implies
\(
T_{\max}=+\infty.
\)
Thus the solution exists globally and satisfies
\begin{align}
g_t\in\mathcal H,
\qquad
\mathcal J(g_t)\le \mathcal J(g_0),
\qquad
\forall \; t\ge0.
\end{align}

\subsection{Convergence of the terminal marginal}
\label{app:terminal-marginal-convergence}

We prove that
\begin{align}
m_{g_t}\to \nu
\qquad
\text{as } t\to\infty.
\end{align}
Since \(t\mapsto\mathcal J(g_t)\) is nonincreasing and bounded from below by
\(\mathcal J(g^\star)\), it has a finite limit as \(t\to\infty\). Integrating
the dissipation identity (\ref{dissipation}) over \([0,T]\) gives
\begin{align}\label{int}
\int_0^T
\left[
\KL(m_{g_t}\|\nu)
+
\KL(\nu\|m_{g_t})
\right]dt
=
\mathcal J(g_0)-\mathcal J(g_T).
\end{align}
Letting \(T\to\infty\) in \eqref{int}, we obtain
\begin{align}\label{intint}
\int_0^\infty
\JSD(g_t)\,dt
<\infty,
\end{align}
where
$
\JSD(g)
=
\KL(m_g\|\nu)
+
\KL(\nu\|m_g).$ 

We claim that \(\JSD(g_t)\to0\) as $t\to+\infty$. Indeed, by the global existence argument, the
trajectory $g_t$ remains in the compact set \(\mathcal C_{g_0}\). Since the vector
field \(F\) is bounded on \(\mathcal C_{g_0}\), the curve
\(t\mapsto g_t\) is uniformly Lipschitz on \([0,\infty)\). Moreover, \(g\mapsto\JSD(g)\) is
continuous (indeed smooth) on \(\mathcal C_{g_0}\), because \(m_g(y)>0\) for
all \(y\in S\). Therefore, \(t\mapsto \JSD(g_t)\) is uniformly continuous on
\([0,\infty)\). Combined with the integrability property \eqref{intint}, this implies that
$\JSD(g_t)\to0.$
In particular, $\KL(m_{g_t}\|\nu)\to0$.
Since \(S\) is finite, this implies
$m_{g_t}\to\nu$ as $t\to +\infty$.

\subsection{Convergence of the potential}
\label{app:potential-convergence}

We now prove that
\begin{align}
g_t\to g^\star
\qquad
\text{as } t\to\infty.
\end{align}
The trajectory \((g_t)_{t\ge0}\) solution to the potential flow~\eqref{eq:SBOF} with initial condition $g_0\in \mathcal H$ remains in the compact set
\(\mathcal C_{g_0}\). Let \((t_n)_{n\ge1}\) be any sequence such that
\(t_n\to\infty\). By compactness, there exists a subsequence, still denoted
\((t_n)\), and some \(\bar g\in\mathcal C_{g_0}\) such that
\(
g_{t_n}\to \bar g.
\)
Since the map \(g\mapsto m_g\) is continuous, and since
\(m_{g_t}\to\nu\), we obtain
\(
m_{\bar g}=\nu.
\)
Moreover, \(\bar g\in\mathcal H\) because \(\mathcal H\) is closed and
\(g_t\in\mathcal H\) for all \(t\ge0\).

By Theorem~\ref{thm:exist-unique}, there is a unique element of
\(\mathcal H\) satisfying the Schrödinger marginal equation \(m_g=\nu\),
namely \(g^\star\). Hence
\(
\bar g=g^\star \).
Thus every accumulation point of the trajectory is \(g^\star\). Since the
trajectory is precompact, this implies that the whole trajectory converges:
\(
g_t\to g^\star,
\) as $t\to +\infty$. 
Finally, by continuity of \(g\mapsto m_g\),
\begin{align*}
m_{g_t}\to m_{g^\star}=\nu.
\end{align*}
This completes the proof of Theorem~\ref{thm:SBOF-convergence}.

\section{Convergence of Couplings and Path Measures}
\label{app:path-convergence}

We prove Corollary~\ref{cor:coupling-path-convergence}.  By Theorem~\ref{thm:SBOF-convergence},
\(g_t\to g^\star\) as \(t\to\infty \).
Since
\begin{align*}
\pi_g(x,y)
=
\mu(x)
\frac{K(x,y)e^{g(y)}}{Z_g(x)}
\end{align*}
and \(Z_g(x)>0\), the map \(g\mapsto \pi_g\) is continuous. Therefore, for every
\(x,y\in S\),
\begin{align*}
\pi_{g_t}(x,y)\to \pi_{g^\star}(x,y).
\end{align*}
Since \(S\times S\) is finite, this implies
$
\|\pi_{g_t}-\pi_{g^\star}\|_{\TV}\to0.
$ By Theorem~\ref{thm:exist-unique}, \(m_{g^\star}=\nu\). Therefore
\(\pi_{g^\star}\in\Pi(\mu,\nu)\), and by the static optimality result,
$
\pi_{g^\star}=\pi^\star.
$
Thus
$
\|\pi_{g_t}-\pi^\star\|_{\TV}\to0.
$

We now prove convergence of the corresponding path measures. Recall that
\begin{align*}
\mathbf P^g(d\omega)
=
\sum_{x,y\in S}
\pi_g(x,y)
\mathbf R_\mu(d\omega\mid X_0=x,X_1=y).
\end{align*}
Therefore
\begin{align}\label{pt-pg}
\mathbf P^{g_t}-\mathbf P^{g^\star}
=
\sum_{x,y\in S}
\left(
\pi_{g_t}(x,y)-\pi_{g^\star}(x,y)
\right)
\mathbf R_\mu(\cdot\mid X_0=x,X_1=y).
\end{align}
Taking total variation in~\eqref{pt-pg} and using the triangle inequality gives
\begin{align}\label{normP-P}
\|\mathbf P^{g_t}-\mathbf P^{g^\star}\|_{\TV}
\le
\sum_{x,y\in S}
\left|
\pi_{g_t}(x,y)-\pi_{g^\star}(x,y)
\right|.
\end{align}
The right-hand side of~\eqref{normP-P} tends to zero. Hence
$
\|\mathbf P^{g_t}-\mathbf P^{g^\star}\|_{\TV}\to0.
$
Since \(\mathbf P^{g^\star}=\mathbf P^\star\), we obtain
$
\|\mathbf P^{g_t}-\mathbf P^\star\|_{\TV}\to0.$
This proves the corollary.

\section{Uniform Poincaré Inequality and Exponential Convergence}
\label{app:poincare-exp}

We provide the details behind the uniform coercivity and exponential
relaxation estimates used in Section~\ref{sec:exponential}.

\subsection{Compact sublevel sets}
\label{app:compact-sublevel-sets}

For \(g_0\in\mathcal H\), we recall that
\begin{align*}
\mathcal C_{g_0}
=
\left\{
g\in\mathcal H:
\mathcal J(g)\le \mathcal J(g_0)
\right\}.
\end{align*}
By the coercivity of \(\mathcal J\) on \(\mathcal H\), proved in
Appendix~\ref{app:coercivity}, the set \(\mathcal C_{g_0}\) is compact.
Moreover, by the dissipation identity \eqref{eq:J-dissipation}, every
solution \(g_t\) of the potential flow \eqref{eq:SBOF} with initial condition \(g_0\in \mathcal H\) satisfies
$
g_t\in\mathcal C_{g_0},$ for all $t\ge0$.

\subsection{Uniform Poincaré inequality}
\label{app:uniform-poincare}

We prove Theorem~\ref{thm:uniform-poincare}. 

For fixed \(g\), define
\begin{align}\label{lambdag}
\lambda(g)
=
\inf
\left\{
\langle f,\mathsf K_g f\rangle:
\mathbb E_\nu[f]=0,\;
\|f\|_{L^2(\nu)}=1
\right\}.
\end{align}
The covariance representation gives
\begin{align}\label{quadraticform}
\langle f,\mathsf K_g f\rangle
=
\sum_{x\in S}\mu(x)\operatorname{Var}_{p_g(\cdot|x)}(f), \qquad f\in \mathbb R^S.
\end{align}
Since \(K(x,y)>0\), each \(p_g(\cdot|x)\) has full support on \(S\). Hence
the quadratic form~\eqref{quadraticform} vanishes if and only if \(f\) is constant. If
\(\mathbb E_\nu[f]=0\), this implies \(f=0\). Therefore
\(
\lambda(g)>0
\)
for every fixed \(g\).

The map \(g\mapsto \mathsf K_g\) is continuous on \(\mathbb R^S\). We now
show that \(g\mapsto \lambda(g)\) is continuous. Let
\begin{align}\label{Snu}
\mathbb S_\nu
=
\left\{
f\in\mathcal H:\|f\|_{L^2(\nu)}=1
\right\}.
\end{align}
Then by~\eqref{lambdag} and~\eqref{Snu}, we can write
\begin{align}
\lambda(g)
=
\inf_{f\in\mathbb S_\nu}
\langle f,\mathsf K_g f\rangle.
\end{align}
For \(g,g'\in\mathbb R^S\) and \(f\in\mathbb S_\nu\),
\begin{align}
\left|
\langle f,\mathsf K_g f\rangle
-
\langle f,\mathsf K_{g'} f\rangle
\right|=& \left|
\langle f,(\mathsf K_g- \mathsf K_{g'})f\rangle
\right|,\notag\\
\le&
\|\mathsf K_g-\mathsf K_{g'}\|_{L^2(\nu)\to L^2(\nu)}.\label{lambdagg'}
\end{align}
Taking infima in~\eqref{lambdagg'} and exchanging the roles of \(g\) and \(g'\) gives
\begin{align}
|\lambda(g)-\lambda(g')|
\le
\|\mathsf K_g-\mathsf K_{g'}\|_{L^2(\nu)\to L^2(\nu)}.
\end{align}
Thus \(g\mapsto\lambda(g)\) defined by~\eqref{lambdag} is continuous.  Since
\(\mathcal C_{g_0}\) is compact, the minimum
\begin{align*}
\lambda_{g_0}
=
\min_{g\in\mathcal C_{g_0}}\lambda(g)=\lambda(\tilde g)
\end{align*}
is attained for some $\tilde g\in \mathcal C_{g_0}$ and hence it is strictly positive. Hence, for every
\(g\in\mathcal C_{g_0}\) and every \(f\) with \(\mathbb E_\nu[f]=0\),
\begin{align*}
\langle f,\mathsf K_g f\rangle
\ge
\lambda_{g_0}\|f\|_{L^2(\nu)}^2.
\end{align*}
This proves the uniform Poincaré inequality~\eqref{eq:PI} of Theorem~\ref{thm:uniform-poincare}. .

\subsection{Entropy--variance comparison}
\label{app:entropy-variance}

We prove Lemma~\ref{lem:entropy-variance}. \\
Since \(g\mapsto m_g\) is
continuous and \(\mathcal C_{g_0}\) is compact, the set
\begin{align*}
\{m_g \in \mathcal P(S):\,g\in\mathcal C_{g_0}\}
\end{align*}
is a compact subset of the interior of the probability simplex  $\mathcal P(S)$. Consequently, there exist
constants \(0<a_{g_0}\le b_{g_0}<\infty\) such that
\begin{align}\label{ab}
a_{g_0}
\le
\frac{m_g(y)}{\nu(y)}
\le
b_{g_0},
\qquad
g\in\mathcal C_{g_0},\ y\in S.
\end{align}
Set
\begin{align*}
r_g=\frac{m_g}{\nu},
\qquad
\ell_g=\log r_g,
\qquad
\bar\ell_g=\mathbb E_\nu[\ell_g].
\end{align*}
Since \(\mathbb E_\nu[r_g]=1\), it follows that
\begin{align}\label{KLmg}
\KL(m_g\|\nu)
=
\mathbb E_\nu[r_g\log r_g]
=
\mathbb E_\nu[\Phi(r_g)],
\end{align}
where
\(
\Phi(r)=r\log r-r+1\).  By Taylor's theorem and~\eqref{ab},
\begin{align}\label{Phir}
\Phi(r)
=
\frac12\Phi''(\xi)(r-1)^2
\le
\frac{1}{2a_{g_0}}(r-1)^2,
\end{align}
for some \(\xi\) between \(r\) and \(1\), since
\(
\Phi''(r)=1/r
\)
and \(r\ge a_{g_0}\). Hence one may choose
\(
C_1=(2a_{g_0})^{-1}.
\)
Therefore, by~\eqref{KLmg} and~\eqref{Phir}, we obtain
\begin{align}\label{KLmgg}
\KL(m_g\|\nu)
\le
C_1\|r_g-1\|_{L^2(\nu)}^2.
\end{align}
Since \(r_g\in[a_{g_0},b_{g_0}]\), the mean value theorem yields
\begin{align}\label{rg-rg}
|r_g(y)-r_g(z)|
\le
b_{g_0}|\ell_g(y)-\ell_g(z)|,
\qquad y,z\in S.
\end{align}
Taking the variance in~\eqref{rg-rg} with respect to \(\nu\), we obtain
\begin{align}\label{Varrg}
\operatorname{Var}_\nu(r_g)
\le
b_{g_0}^2\operatorname{Var}_\nu(\ell_g).
\end{align}
Since \(\mathbb E_\nu[r_g]=1\), we have $
\operatorname{Var}_\nu(r_g)
=
\|r_g-1\|_{L^2(\nu)}^2.
$
Hence, \eqref{Varrg} can be written as
\begin{align}\label{rgrgrg}
\|r_g-1\|_{L^2(\nu)}^2
\le
b_{g_0}^2
\|\ell_g-\bar\ell_g\|_{L^2(\nu)}^2.
\end{align}
Combining \eqref{KLmgg} and  \eqref{rgrgrg} gives
\begin{align*}
\KL(m_g\|\nu)
\le
C_{g_0}
\left\|
\log\frac{m_g}{\nu}
-
\mathbb E_\nu\!\left[
\log\frac{m_g}{\nu}
\right]
\right\|_{L^2(\nu)}^2
\end{align*}
for some constant \(C_{g_0}>0\). This proves Lemma~\ref{lem:entropy-variance}.

\subsection{Proof of Theorem~\ref{thm:exponential-convergence}}
\label{app:proof-exponential-convergence}

Let \(g_t\) be the solution of potential flow  \eqref{eq:SBOF} with initial condition $g_0\in \mathcal H$, and set
\begin{align*}
m_t=m_{g_t},
\qquad
h_t=\log\frac{m_t}{\nu}.
\end{align*}
By Proposition~\ref{prop:marginal-dissipation},
\begin{align}\label{klklklk}
\frac{d}{dt}\KL(m_t\|\nu)
=
-\langle h_t,\mathsf K_{g_t}h_t\rangle.
\end{align}
Since \(\mathsf K_{g_t}\mathbf 1=0\), we may subtract the \(\nu\)-mean of
\(h_t\) in~\eqref{klklklk}:
\begin{align}\label{h_tkgt}
\langle h_t,\mathsf K_{g_t}h_t\rangle
=
\left\langle
h_t-\mathbb E_\nu[h_t],
\mathsf K_{g_t}
\left(h_t-\mathbb E_\nu[h_t]\right)
\right\rangle.
\end{align}
Since \(g_t\in\mathcal C_{g_0}\), the uniform Poincaré inequality~\eqref{eq:PI}  with~\eqref{h_tkgt} gives
\begin{align}\label{cg0}
\langle h_t,\mathsf K_{g_t}h_t\rangle
\ge
\lambda_{g_0}
\left\|
h_t-\mathbb E_\nu[h_t]
\right\|_{L^2(\nu)}^2.
\end{align}
Moreover, by the entropy--variance comparison gives
\begin{align}\label{entvar}
\left\|
h_t-\mathbb E_\nu[h_t]
\right\|_{L^2(\nu)}^2
\ge
\frac{1}{C_{g_0}}
\KL(m_t\|\nu).
\end{align}
Combining  estimates ~\eqref{klklklk}, \eqref{klklklk} and \eqref{cg0} yields
\begin{align}
\frac{d}{dt}\KL(m_t\|\nu)
\le
-
\frac{\lambda_{g_0}}{C_{g_0}}
\KL(m_t\|\nu).
\end{align}
Applying Grönwall's inequality, we obtain the desired exponential relation~\eqref{eq:exp-convergence}.
This proves the theorem.

\section{Doob Transform and Numerical Details}
\label{app:numerics-doob}
This appendix provides the technical details underlying the dynamic
reconstruction and numerical illustration presented in Section~\ref{sec:dynamic-reconstruction}.
The first part shows that the Doob transform associated with a terminal
potential \(g\) yields the path measure obtained by lifting the endpoint
coupling \(\pi_g\) via the reference Markov bridges. The second part
summarizes the numerical scheme and the diagnostics used to evaluate the potential flow.

\subsection{Path-space reconstruction by Doob transform}
\label{app:doob-transform}

Let $g:S\to\mathbb{R}$, and define
\begin{equation}\label{phis}
\varphi_s^g=e^{(1-s)Q}e^g,
\qquad s\in[0,1].
\end{equation}
Then $\varphi_s^g(x)>0$ for all $s\in[0,1]$ and $x\in S$. Moreover, the function
$\varphi_s^g$ defined in~\eqref{phis} is the unique solution to the backward equation
\begin{equation*}
\partial_s\varphi_s^g=-Q\varphi_s^g,
\end{equation*}
with terminal condition $\varphi_{s=1}^g=e^g$.
For $x\neq y$, define
\begin{equation*}
Q_s^g(x,y)
=
Q(x,y)\frac{\varphi_s^g(y)}{\varphi_s^g(x)},
\end{equation*}
and choose the diagonal entries $Q_s^g(x,x)=-\sum_{y\neq x}Q_s^g(x,y)$ so that each row sums to zero. The family
$(Q_s^g)_{s\in[0,1]}$ constitutes the space-time Doob transform of the reference
generator $Q$, a classical construction in the theory of Markov bridges and SB; see, for instance,
\cite{doob1957conditional,jamison1975markov,leonard2014survey}.

We next verify that the resulting time-inhomogeneous Markov chain has endpoint
conditional law $p_g(\cdot|x)$. Let $X_t$ denote the state of the reference Markov process at time $t \in [0,1]$. By the Markov property and the
definition of $\varphi_s^g$ in~\eqref{phis},
\begin{equation}
\mathbb{E}_{\mathbf{R}}\left[e^{g(X_1)}\mid X_s=x\right]
=
\varphi_s^g(x).
\end{equation}
Let $\mathbf{P}_x^g$ denote the law of the Doob-transformed process initialized at $X_0=x$. The corresponding path measure on $[0,1]$ is therefore given with respect to the reference measure $\mathbf{R}_x$ by
\begin{equation*}
\frac{d\mathbf{P}_x^g}{d\mathbf{R}_x}
=
\frac{e^{g(X_1)}}{\varphi_0^g(x)}.
\end{equation*}
It follows that
\begin{equation*}
\mathbf{P}_x^g(X_1=y)
=
\frac{\mathbf{R}_x(X_1=y)e^{g(y)}}{\varphi_0^g(x)}
=
\frac{K(x,y)e^{g(y)}}{\sum_z K(x,z)e^{g(z)}}
=
p_g(y|x).
\end{equation*}
Hence, if $X_0\sim\mu$, the corresponding endpoint coupling is
\begin{equation*}
\mu(x)p_g(y|x)=\pi_g(x,y).
\end{equation*}
It remains to identify the conditional bridges. Since the Radon--Nikodym
derivative $e^{g(X_1)}/\varphi_0^g(X_0)$ depends only on the endpoints,
conditioning on $(X_0,X_1)=(x,y)$ cancels out this factor. Consequently, the
Doob-transformed process and the reference process share the same conditional
law given the endpoints:
\begin{equation*}
\mathbf{P}^g(\cdot\mid X_0=x,X_1=y)
=
\mathbf{R}_\mu(\cdot\mid X_0=x,X_1=y).
\end{equation*}
Therefore, integrating over the initial and terminal states yields
\begin{equation*}
\mathbf{P}^g(d\omega)
=
\sum_{x,y\in S}
\pi_g(x,y)
\mathbf{R}_\mu(d\omega\mid X_0=x,X_1=y).
\end{equation*}
In particular, when $g=g^\star$, we have $m_{g^\star}=\nu$, and the
corresponding path measure exactly coincides with the SB path
measure.

\subsection{Numerical implementation details}
\label{app:implementation-details}

We describe the numerical choices used in Section~\ref{sec:dynamic-reconstruction}.
The state space is a finite two-dimensional grid $S$, and the reference
generator $Q$ is chosen as a local continuous-time random-walk generator on
this grid. The terminal transition kernel is given by
$
K=e^Q.
$
All computations are performed in finite dimension. The positivity assumption
is ensured by choosing an irreducible generator and a strictly positive time
horizon.

The potential flow in~\eqref{eq:SBOF} is discretized on a finite time interval
$[0,T]$. Given $N_{\rm it}\in\mathbb{N}^*$, we set
\begin{equation}
\Delta t=\frac{T}{N_{\rm it}},
\qquad
t_k=k\Delta t,
\qquad
k=0,\ldots,N_{\rm it}.
\end{equation}
Starting from $g_0\in\mathcal{H}$, we compute the approximations
$g_k\approx g_{t_k}$ using the explicit Euler scheme
\begin{equation}\label{ggggg}
g_{k+1}
=
g_k
-
\Delta t
\left(
\log\frac{m_{g_k}}{\nu}
-
\mathbb{E}_\nu\!\left[
\log\frac{m_{g_k}}{\nu}
\right]\mathbf{1}
\right).
\end{equation}
After each iteration~\eqref{ggggg}, we project onto the gauge-fixed space by setting
\begin{equation}\label{gk+1}
g_{k+1}
\leftarrow
g_{k+1}
-
\mathbb{E}_\nu[g_{k+1}]\mathbf{1}.
\end{equation}
The projection step in~\eqref{gk+1} merely removes numerical drift in the additive gauge
direction and leaves both the coupling $\pi_{g_{k+1}}$ and the marginal
$m_{g_{k+1}}$ unchanged.

The endpoint coupling is then given by
\begin{equation*}
\pi_{g_k}(x,y)
=
\mu(x)
\frac{K(x,y)e^{g_k(y)}}{\sum_{z\in S}K(x,z)e^{g_k(z)}}.
\end{equation*}
As a high-accuracy reference, we run the classical Sinkhorn/IPF iterations
until the marginal residual falls below the prescribed tolerance. We then
report the terminal relative entropy, total variation error, maximal
logarithmic error, and coupling error:
\begin{equation*}
\KL(m_{g_T}\|\nu),
\qquad
\|m_{g_T}-\nu\|_{\TV},
\qquad
\max_{y\in S}\left|\log\frac{m_{g_T}(y)}{\nu(y)}\right|,
\qquad
\|\pi_{g_T}-\pi^\star\|_{\TV}.
\end{equation*}
The Onsager entropy-production diagnostic is evaluated through the covariance
formula
\begin{equation*}
\mathcal{I}_{g_k}(m_{g_k}\|\nu)
=
\sum_{x\in S}
\mu(x)
\operatorname{Var}_{p_{g_k}(\cdot|x)}
\left(
\log\frac{m_{g_k}}{\nu}
\right),
\end{equation*}
thereby avoiding the explicit construction of the full matrix $\mathsf{K}_{g_k}$.

\end{document}